\documentclass[12pt,a4paper,reqno]{amsart}

\usepackage[margin=2.8cm,footskip=1cm]{geometry}
\usepackage[utf8]{inputenc}
\usepackage[T1]{fontenc}
\usepackage[english]{babel}
\usepackage{amsmath}
\usepackage{amsfonts}
\usepackage{amssymb}
\usepackage{amsthm}
\usepackage{ucs}
\usepackage{graphicx}
\usepackage{xcolor}
\usepackage{dsfont}
\usepackage{tikz}
\usepackage{wasysym}
\usepackage{physics}
\usepackage{mathtools}
\usepackage{xifthen}
\usepackage[textwidth=2.5cm,textsize=tiny]{todonotes}
\usepackage{enumitem}
\usepackage{stmaryrd}
\usepackage{constants}
\usepackage[ruled]{algorithm2e}
\usepackage{tikz,pgfplots}
\usepackage{hyperref}
\usepackage[font={small}]{caption}

\graphicspath{{Pictures/}}

\usepackage{constants}

\newconstantfamily{ctrlcst}{
	symbol=\kappa
}
\newconstantfamily{csts}{
	symbol=C
}
\renewconstantfamily{normal}{
	symbol=c
}

\makeatletter
\newcommand{\pushright}[1]{\ifmeasuring@#1\else\omit\hfill$\displaystyle#1$\fi\ignorespaces}
\newcommand{\pushleft}[1]{\ifmeasuring@#1\else\omit$\displaystyle#1$\hfill\fi\ignorespaces}
\makeatother

\newcommand{\bbZ}{\mathbb{Z}}
\newcommand{\bbR}{\mathbb{R}}
\newcommand{\R}{\mathbb{R}}

\newcommand{\Ztwo}{\mathbb{Z}^2}

\newcommand{\betac}{\beta_{\mathrm{\scriptscriptstyle c}}}
\renewcommand{\norm}[1]{\|#1\|}

\newcommand{\setof}[2]{\{#1\,:\,#2\}}
\newcommand{\bsetof}[2]{\bigl\{#1\,:\,#2\bigr\}}
\newcommand{\given}{\,|\,}
\newcommand{\bgiven}{\bigm\vert}

\newcommand{\PottsZ}{\mathcal{Z}}

\newcommand{\eone}{\vec e_1}

\newcommand{\Z}{\mathbb{Z}}

\newcommand{\IF}[1]{\mathds{1}_{\{#1\}}}

\renewcommand{\emptyset}{\varnothing}
\newcommand{\calS}{\mathcal{S}}

\newcommand{\HausdorffDist}{d_{\mathrm{H}}}

\newcommand{\excursion}{\mathfrak{e}}
\newcommand{\iid}{i.i.d.\xspace}

\newcommand{\fcone}{\mathcal{Y}^\blacktriangleleft}
\newcommand{\bcone}{\mathcal{Y}^\blacktriangleright}
\newcommand{\bend}{\mathbf{b}}
\newcommand{\fend}{\mathbf{f}}
\newcommand{\diam}{D}

\newcommand{\concatenate}{\circ}
\newcommand{\displace}{X}

\newcommand{\walk}{\mathsf{S}}

\newcommand{\bfp}{\mathbf{p}}

\newcommand{\SetRootMarkBackCont}{\mathfrak{B}_L}
\newcommand{\SetRootMarkForwCont}{\mathfrak{B}_R}
\newcommand{\SetRootDiaCont}{\mathfrak{A}}

\newcommand{\CPts}{\textnormal{CPts}}

\newcommand{\PottsLaw}{\mu}

\newcommand{\FKlaw}{\Phi}
\newcommand{\FKZ}{Z}

\newcommand{\open}{\textnormal{\xspace open}}
\newcommand{\close}{\textnormal{\xspace closed}}
\newcommand{\babs}[1]{\bigl\lvert #1 \bigr\rvert}
\newcommand{\Babs}[1]{\Bigl\lvert #1 \Bigr\rvert}

\newcommand{\coupling}{\Psi}
\newcommand{\Cov}{\mathrm{Cov}}
\newcommand{\ouvert}{\mathrm{o}}
\newcommand{\Piv}{\mathrm{Piv}}

\newcommand{\EdgeBnd}{\partial_{\scriptscriptstyle\rm edge}}

\theoremstyle{plain}
\newtheorem{theorem}{Theorem}[section]
\newtheorem{lemma}[theorem]{Lemma}
\newtheorem{proposition}[theorem]{Proposition}
\newtheorem{corollary}[theorem]{Corollary}

\newtheorem{definition}{Definition}[section]
\newtheorem{remark}{Remark}[section]

\newtheorem{claim}{Claim}

\theoremstyle{definition}
\newtheorem{obs}{Observation}

\author{Dmitry Ioffe}
\address{Faculty of IE\&M, Technion, Haifa 32000, Israel}
\email{ieioffe@ie.technion.ac.il}
\author{S\'{e}bastien Ott}
\address{Dipartimento di Matematica e Fisica, Università degli Studi Roma Tre, 00146 Roma, Italy}
\email{ott.sebast@gmail.com}
\author{Yvan Velenik}
\address{Section de Mathématiques, Université de Genève, CH-1211 Genève, Switzerland}
\email{yvan.velenik@unige.ch}
\author{Vitali Wachtel}
\address{Institut f\"ur Mathematik, Universit\"at Augsburg, D-86135 Augsburg, 
    Germany}
\email{vitali.wachtel@math.uni-augsburg.de}

\newcommand{\calL}{\mathcal{L}}

\newcommand{\fre}{\mathfrak{e}}

\newcommand{\frz}{\mathfrak{z}}
\newcommand{\frA}{\mathfrak{A}}
\newcommand{\frB}{\mathfrak{B}}

\newcommand{\frI}{\mathfrak{I}}

\newcommand{\frT}{\mathfrak{T}}

\newcommand{\bbE}{\mathbb{E}}

\newcommand{\bbH}{\mathbb{H}}

\newcommand{\bbN}{\mathbb{N}}

\newcommand{\bfE}{\mathbf{E}}

\newcommand{\bfG}{\mathbf{G}}

\newcommand{\bfK}{\mathbf{K}}

\newcommand{\bfQ}{\mathbf{Q}}

\newcommand{\bfZ}{\mathbf{Z}}

\newcommand{\sfa}{\mathsf a}
\newcommand{\sfb}{\mathsf b}

\newcommand{\sfe}{\mathsf e}

\newcommand{\sfo}{{\mathsf o}}
\newcommand{\sfp}{{\mathsf p}}

\newcommand{\sfu}{{\mathsf u}}
\newcommand{\sfv}{{\mathsf v}}
\newcommand{\sfw}{{\mathsf w}}
\newcommand{\sfx}{{\mathsf x}}
\newcommand{\sfy}{{\mathsf y}}
\newcommand{\sfz}{{\mathsf z}}

\newcommand{\sfE}{\mathsf{E}}

\newcommand{\sfH}{\mathsf{H}}

\newcommand{\sfO}{\mathsf{O}}
\newcommand{\sfP}{\mathsf{P}}

\newcommand{\sfS}{\mathsf{S}}
\newcommand{\sfT}{\mathsf{T}}

\newcommand{\sfZ}{\mathsf{Z}}

\newcommand{\ueta}{\underline{\eta}}
\newcommand{\ulambda}{\underline{\lambda}}

\newcommand{\defby}{=}

\newcommand{\lb}{\left(}
\newcommand{\rb}{\right)}
\newcommand{\lbr}{\left\{}
\newcommand{\rbr}{\right\}}

\newcommand{\be}[1]{\begin{equation}\label{#1}}
\newcommand{\ee}{\end{equation}}

\date{\today}

\title[Invariance principle for a Potts interface along a wall]{Invariance principle for\\a Potts interface along a wall}

\begin{document}

\maketitle	
\begin{abstract}
We consider nearest-neighbour two-dimensional Potts models, with boundary conditions leading to the presence of an interface along the bottom wall of the box. We show that, after a suitable diffusive scaling, the interface weakly converges to the standard Brownian excursion.
\end{abstract}


\section{Introduction and results}

The rigorous understanding of the statistical properties of interfaces in two-dimensional spin systems has raised considerable interest for nearly 50 years.

Early results mostly dealt with the very-low temperature Ising model. The first rigorous result indicating diffusive behavior for the interface in this model was obtained by Gallavotti in 1972~\cite{Gallavotti-1972}. It was shown in this paper that, at sufficiently low temperature, the interface in a box of linear size \(n\) has fluctuations of order \(\sqrt{n}\). A description of the internal structure of the interface (in particular the fact that the interface has a bounded intrinsic width, in spite of its unbounded fluctuations) was provided in~\cite{Bricmont+Lebowitz+Pfister-1981}, while a full invariance principle toward a Brownian bridge was proved in~\cite{Higuchi-1979}. These works were completed by a number of (nonperturbative) exact results in which the profile of expected magnetization was derived in the presence of an interface, see for instance~\cite{Abraham+Reed-1974}. Extensions of such low-temperature results to other two-dimensional models have been obtained, although a complete theory is still lacking.

The absence of tools to undertake a nonperturbative analysis led to the analysis of similar problems in simpler ``effective'' settings; see, for instance, \cite{Durrett-1979}.

Nevertheless, during the last 20 years, a lot of progress has been made toward extending such results to all temperatures below critical. In particular, a detailed description of the microscopic structure of the interface as well as a proof of an invariance principle were provided in~\cite{Campanino+Ioffe+Velenik-2003,Greenberg+Ioffe-2005} for the Ising model and~\cite{Campanino+Ioffe+Velenik-2008} for the Potts model.

All the above results were concerned with an interface ``in the bulk'' (that is, an interface crossing an ``infinite strip''). For a long time, the understanding of the corresponding properties for an interface located along one of the system's boundaries remained much more elusive, even in perturbative regimes. The difficulty is that one has to understand how the interface interacts with the boundary and, in particular, exclude pinning of the interface by the wall. It turns out that a rigorous understanding of such issues requires a surprisingly careful analysis. This was undertaken, in a perturbative regime, by Ioffe, Shlosman and Toninelli in~\cite{Ioffe+Shlosman+Toninelli-2015}. Although restricted to Ising-type interface, the approach they develop is in principle of a rather general nature.

In~\cite{Dobrushin-1992}, Dobrushin states convergence of a properly rescaled Ising interface above a wall towards the standard Brownian excursion, for sufficiently low temperatures.
The proof is briefly sketched with a reference to the fundamental low-temperature techniques developed in~\cite{Dobrushin+Kotecky+Shlosman-1992}. 
It is not entirely clear whether a complete rigorous implementation along these lines would indeed follow from 
the results in~\cite[Chapter~4]{Dobrushin+Kotecky+Shlosman-1992} alone (with the simple correction presented in the Appendix of~\cite{Ioffe+Shlosman+Toninelli-2015}) or whether it would require the full power of~\cite{Ioffe+Shlosman+Toninelli-2015} in order to control the competition between the entropic repulsion and the interaction between  the interface and the wall. 

In the present paper, we prove that such an interface, after suitable diffusive scaling, converges to a Brownian excursion, for all temperatures below \(T_c\) and arbitrary \(q\)-state Potts models. We bypass a detailed analysis of the interaction between the interface and the wall by combining monotonicity and mixing properties of these models. 
{Lemma~\ref{lem:wall_to_energy_cost},  which should be considered as one of the main technical, and perhaps conceptual, contributions of this paper,  implies that in the case of nearest neighbor Potts models on $\bbZ^2$, entropic repulsion of the interface from the wall wins over a  possible attraction of the interface by the wall for all temperatures below critical. This result has important ramifications, for instance it plays a crucial role for proving convergence to Ferrari--Spohn diffusions of low-temperature Ising interfaces in the critical prewetting regime \cite{IOSV}}, or for studying low-temperature 2D Ising metastable states related to the phenomenon of uphill diffusions \cite{DIMP}.

\begin{figure}[t]
	\begin{minipage}{.65\textwidth}
		\centering
		\includegraphics[width=.95\textwidth]{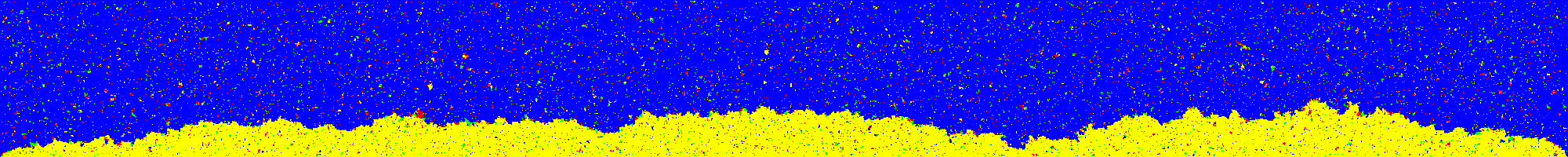}  \\[1cm]
		\includegraphics[width=.95\textwidth]{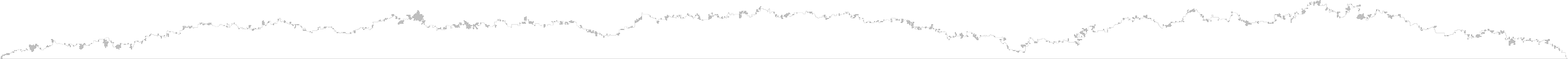}
	\end{minipage}
	\begin{minipage}{.31\textwidth}
	\centering
	\includegraphics[width=.95\textwidth]{Pic/ScalingLimit.pdf}
	\end{minipage}
	\caption{Top left: typical Potts configuration. Bottom left: the corresponding interface. Right: The interface after diffusive scaling.}
\end{figure}

\subsection{Notations and Conventions}

We denote \(\Z_+=\{0,1,2,\dots\}\) the non-negative integers. 
{\( C , C_1, \dots , c,c_1 , \dots\)}  will denote non-negative constants whose value can change from line to line and that do not depend on the parameters under investigation.

Denote $\Z^2=(V_{\Z^2}\equiv V,E_{\Z^2}\equiv E)$ the graph with vertices \(\bsetof{i=(i_1,i_2)\in\R^2}{i_1,i_2\in\Z}\) and edges between any two vertices \(i,j\) at Euclidean distance $1$, which we denote by \(i\sim j\).
The dual graph $(\Z^2)^*=(V^*,E^*)$ has set of vertices {$V+(1/2,1/2)$ and edges between any two vertices at distance $1$}.
There is a natural bijection between $E$ and $E^*$, mapping the edge \(e=\{i,j\}\in E\) to the unique edge \(e^*=\{i,j\}^*\in E^*\) intersecting it; we then say that \(e\) and \(e^*\) are dual to each other.

It will be convenient to see a set $C\subset E$ both as a set of edges and as the subset of $\R^2$ given by the union of the closed line segments defined by the edges. We will say that a vertex belongs to $C$ if it is an endpoint of at least one edge of \(C\).  We denote by $\EdgeBnd C$ the set of edges in $\Z^2\setminus C$ having at least one endpoint in $C$.
Those conventions are adapted in a straightforward fashion to $C\subset E^*$.

We will say that two vertices \(u,v\) are \emph{connected} in a graph if there exists a path of edges linking them. We denote this property \(u\leftrightarrow v\).

\subsection{Potts and Random-Cluster Model, Duality}

Let \(q\geq 2\) be an integer, \(\beta\geq 0\), \(G=(V_G,E_G)\) be a graph, \(F=(V_F,E_F)\subset G\) be finite and {\(\alpha\in\{1,\dots,q\}^{V_G}\)}. The \(q\)-state Potts model on \(F\) at inverse temperature \(\beta\) with {boundary condition \(\alpha\) 
 is the probability measure \(\PottsLaw_{\beta,q,F}^{\alpha}\) on \(\{1,\dots,q\}^{V_F}\) defined by
\begin{equation*}
	\PottsLaw_{\beta,q,F}^{\alpha}(\sigma) = \frac{1}{\PottsZ_{\beta,q,F}^{\alpha}} \exp\Bigl(\beta\sum_{\{i,j\}\in E_F} \IF{\sigma_i=\sigma_j} + \beta\sum_{\substack{\{i,j\}\in E_G\\i\in V_F, j\notin V_F}} \IF{\sigma_i=\alpha_j} \Bigr),
\end{equation*}
where \(\PottsZ_{\beta,q,F}^{\alpha}\) is the normalizing constant.

Let \(\beta, G, F\) be as before and \(q\geq 1\) be real. Let \(\eta\in\{0,1\}^{E_G}\). The random-cluster measure on \(F\) with edge weight \(e^{\beta}-1\), cluster weight \(q\) and boundary condition \(\eta\) is the probability measure on \(\{0,1\}^{E_F}\) (identified with the subsets of \(E_F\)) given by
\begin{equation*}
\FKlaw_{\beta,q,F}^{\eta}(\omega) = \frac{1}{\FKZ_{\beta,q,F}^{\eta}} (e^{\beta}-1)^{|\omega|} q^{\kappa_{\eta}(\omega)},
\end{equation*}where \(\kappa_{\eta}(\omega)\) is the number of connected components (clusters) intersecting \(V_F\) in the graph obtained by taking the graph with vertex set \(V_G\) and edge set \((\eta\setminus E_F) \cup \omega\). When omitted from the notation, \(\eta\) is assumed to be identically \(0\) (free boundary conditions).
If the graph \(G\) is taken to be \(\Z^2\), one can define the random-cluster measure dual to \(\FKlaw_{\beta,q,F}^{\eta}\) using the bijection from \(\{0,1\}^{E}\) to \(\{0,1\}^{E^*}\) induced by \(\omega^*_{e^*} = 1-\omega_e\). The dual measure is then \(\FKlaw_{\beta^*,q,F^*}^{\eta^*}\) where \(\beta^*\) is defined via
\begin{equation}
\label{eq:bbstar}
	(e^{\beta}-1)(e^{\beta^*}-1) = q.
\end{equation}
If \(\omega\sim\FKlaw_{\beta,q,F}^{\eta}\), then \(\omega^*\sim \FKlaw_{\beta^*,q,F^*}^{\eta^*}\) (see \cite{Grimmett-2006}).

As the transition temperature of the Potts model on \(\Z^2\) is given by \(\betac = \log(1+\sqrt{q})\) (the self dual point in the sense of~\eqref{eq:bbstar},  see~\cite{Beffara+Duminil-Copin-2012}), one has that \(\beta>\betac\implies \beta^*<\betac\) and vice versa. Moreover, the transition is sharp: for all \(q\geq 1\) and \(\beta<\betac(q)\), there exist \(C,c>0\) such that \(\FKlaw^1_{\beta,q,B_n}(0 \leftrightarrow \partial B_n) \leq Ce^{-cn}\) for all \(n\geq 1\), where \(B_n = \{-n,\dots,n\}^2\).

One main advantage of the random-cluster model is that it satisfies the FKG lattice condition.
The following classical notion will be important for us.
An edge \(e\) is said to be \emph{pivotal} for the event \(A\) in the configuration \(\omega\) if \(\mathds{1}_{A} (\omega) + \mathds{1}_{A} (\omega^\prime ) =1\), 
where the configuration \(\omega'\) is given by \(\omega'_f = \omega_f\) for all \(f\neq e\) and \(\omega'_e = 1-\omega_e\). We denote by \(\Piv_\omega(A)\) the set of all edges that are pivotal for \(A\) in \(\omega\). When averaging over \(\omega\) under some probability measure, we will often simply write \(\Piv(A)\) for the corresponding set of edges.

\subsection{Edwards--Sokal Coupling for Interfaces.}

We are interested in the behavior of the interface between a pure phase occupying the bulk of the system and a second pure phase located along the boundary. It will be convenient to define the Potts model on \((\Z^2)^*\). Denote \(\Lambda_+^*\equiv\Lambda_+^*(N)=\bigl([-N+1/2,N-1/2]\times[-1/2,N-1/2]\bigr)\cap(\Z^2)^*\). We consider the Potts model on \(\Lambda_+^*\) with boundary condition
\begin{equation*}
	\alpha^{\pm}_i= \begin{cases}
	1 & \text{if } i_2<0,\\
	2 & \text{if } i_2>0.
	\end{cases}
\end{equation*}
\(\PottsLaw_{\beta^{*},q, \Lambda_{+}^{*} }^{\alpha^{\pm} } \) is related to the random-cluster model via the Edwards--Sokal coupling: from a configuration \(\sigma\in\{1,\dots,q\}^{\Lambda_{+}^{*}}\), one obtains a configuration \(\omega^{*}\) on \(E^{*}\) by setting 
(here \(e^{*}=\{i,j\}\in E^{*}\) and intersections are between sets of vertices)
\begin{itemize}
	\item \(\omega^{*}_{e^{*}}=1\) if \(\{i,j\}\cap \Lambda_+^* = \varnothing \),
	\item \(\omega^{*}_{e^{*}}=0\) if \(\{i,j\}\subset \Lambda_+^{*} \) and \(\sigma_i\neq \sigma_j\),
	\item \(\omega^{*}_{e^{*}}=0\) if \(\{i,j\}\cap \Lambda_+^{*}=\{i\} \) and \(\sigma_i\neq \alpha_j\),
	\item \(\omega^{*}_{ e^{*} }=\xi_{ e^{*} }\) in the other cases, where \((\xi_{ e^{*} } )_{ e^{*}\in E^{*} }\) is a family of i.i.d.\ Bernoulli random variables of parameter \(1-e^{-\beta}\).
\end{itemize}
Define then \(\omega\in\{0,1\}^{E}\) from \(\omega^*\) by \(\omega_e=1-\omega^*_{e^*}\). One has \(\omega\sim \FKlaw_{\beta,q,\Lambda_+}(\,\cdot \given v_L\leftrightarrow v_R)\) where \(\Lambda_+ = \{-N,\dots,N\}\times\{0,\dots,N\}\) and \(v_L=(-N,0), v_R=(N,0)\). We will also denote \(\Lambda_-=\{-N,\dots,N\}\times\{-1,\dots,-N\}\) .
\begin{figure}[h]
	\centering
	\includegraphics[scale=0.8]{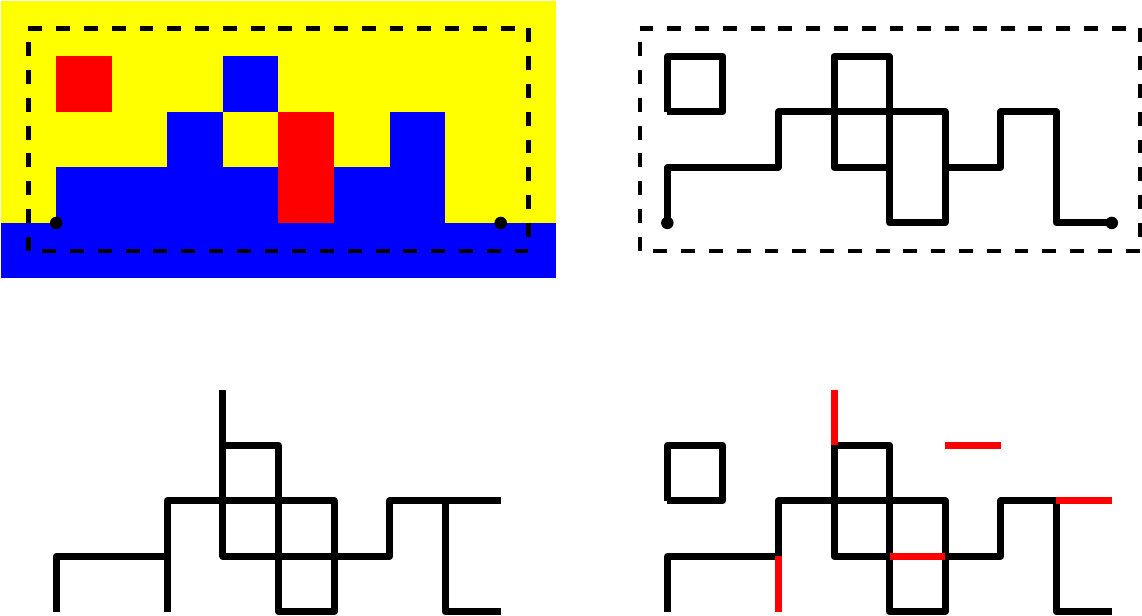}
	\caption{From top left to bottom left in clockwise order: a Potts interface on the dual box; its Peierls contours; the random cluster configuration obtained from it by independently opening nonfrozen edges with probabiity \(1-e^{-\beta}\); the cluster we will study.}
	\label{fig:InterfaceToCluster}
\end{figure}

\begin{remark}
	The way we constructed \(\omega\) implies that the Peierls contours between different colors in the Potts configuration are included in \(\omega\). Thus, any reasonable notion of the interface between \(1\) and \(2\) induced by the boundary condition is a subset of the common cluster of \(v_L\) and \(v_R\) in \(\omega\).
\end{remark}

From now on, we will often omit \(q\) from the notation (it will be supposed integer and \(\geq 2\) when talking about the Potts model and its coupling with the random-cluster model and supposed real and \(\geq 1\) when talking about the random-cluster model alone). We will also systematically take \(\beta^*>\betac(q)>\beta\) and denote by \(\FKlaw\) the (unique) infinite-volume measure. {To lighten notations, we will drop the \(\beta\)-dependency in the proofs (Sections~\ref{sec:OZinfVol_Consequences},~\ref{sec:Entropic_repulsion} and~\ref{sec:Effective_HS_RW}).}

\subsection{Surface Tension and Wulff Shape}
\label{sub:st-Wsh}
For a direction \(s\in \mathbb{S}^1\), define the configuration \({\alpha}^s\in\{1,2\}^{V^*}\) (remember that \(V^*\) is the set of vertices of the graph \((\Ztwo)^*\)) by
\begin{equation*}
	{\alpha}^s_i = \begin{cases}
	1 & \text{ if } i\cdot s>0\\
	2 & \text{else}
	\end{cases},
\end{equation*} where \(\cdot\) denotes the scalar product. The \emph{surface tension} in the direction \(s\) at inverse temperature \(\beta^*\) is defined as
\begin{equation*}
	\tau_{\beta^*}(s) = -\lim_{N\to\infty} \frac{1}{l_s(N)} \log(\frac{\PottsZ_{\beta^*,\Lambda_N^*}^{{\alpha}^s}}{\PottsZ_{\beta^*,\Lambda_N^*}^{1}}),
\end{equation*}
where \(\Lambda_N^*=\bigl([-N+1/2,N-1/2]\times[-N+1/2,N-1/2]\bigr)\cap(\Z^2)^*\) and \(l_s(N)\) is the length of the line segment determined by the intersection of the straight line through \(0\) with normal \(s\) and the set \([-N,N]^2\).
It is known that \(\tau_{\beta^*}(s)>0\) for all \(s\) and all \(\beta^*>\betac(q)\)~\cite{Duminil-Copin+Manolescu-2016}. In fact, the surface tension can be defined for a rather large class of models in arbitrary dimensions~\cite{miracle1995surface} and its homogeneous of order one extension is convex and, therefore, can be represented as the support function of the so-called equilibrium crystal (Wulff) shape ${\mathbf K}_{\beta^*}$. In two dimensions, the boundary $\partial {\mathbf K}_{\beta^*}$ is analytic and has a uniformly positive curvature~\cite{Campanino+Ioffe+Velenik-2008} at all sub-critical temperatures $\beta^* > \beta_c$. The inverse transition temperature {\(\betac = \betac (q) = \log(1+\sqrt{q})\)} can thus be {characterized}  as
\[
\betac(q) = \inf\setof{\beta^*\geq 0}{\tau_{\beta^*}>0}.
\]
Set \(\tau = \tau_{\beta^*}\lb{\eone} \rb \) to be the surface tension in the horizontal axis direction $\eone = (1,0)$.
In the sequel, we shall use $\chi = \chi_{\beta^*}$ to denote the curvature of ${\mathbf K}_{\beta^*}$ at its rightmost point $\tau{\eone}  \in \partial{\mathbf K}_{\beta^*}$.

A direct consequence of the correspondence between the Potts model on \((\Z^2)^*\) at inverse temperature \(\beta^*\) and the random-cluster model on \(\Z^2\) at inverse temperature \(\beta\) is that
\begin{equation*}
\tau = -\lim_{N\to\infty} \frac{1}{2N+1} \log\FKlaw_{\beta,\Lambda_N}(v_L\leftrightarrow v_R) = -\lim_{N\to\infty} \frac{1}{N} \log\FKlaw_{\beta}\big(0\leftrightarrow (N,0)\big),
\end{equation*}
where \(\FKlaw_{\beta}\) is the random-cluster distribution on \(\Ztwo\) obtained as the limit of the finite-volume measures on square boxes with \(0\) boundary condition.

\subsection{Results}

We will denote \(\Gamma= C_{v_L,v_R}\) the joint cluster of \(v_L,v_R\) under \(\FKlaw_{\beta,\Lambda_+}(\,\cdot \given v_L\leftrightarrow v_R)\). We also define the upper and lower vertex boundary of \(\Gamma\): 
\[
\Gamma^+_k=\max\{j:(k,j)\in\Gamma\} \ {\rm and}\   \Gamma^-_k=\min\{j:(k,j)\in\Gamma\}\  {\rm for}\  k=-N,\dots,N .
\]
We will see \(\Gamma^{+}\) and \(\Gamma^{-}\) as integer-valued random functions on $\{-N, \dots, N\}$. 

\begin{figure}[h]
	\centering
	\includegraphics[scale=0.8]{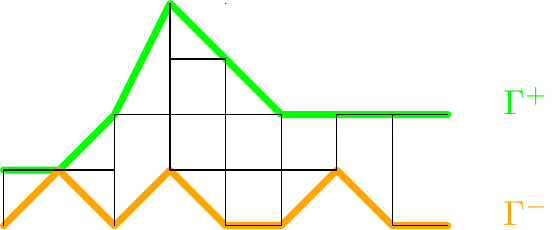}
	\caption{The cluster of Figure~\ref{fig:InterfaceToCluster} and the graphs of the (linear interpolation of the) two associated vertex boundaries \(\Gamma^+\) and \(\Gamma^-\).}
\end{figure}
 
\subsection{Scaling limit of the interface}

Let, for \(t\in [0,1]\),
\begin{equation}\label{eq:GammaScaled}
\hat\Gamma^+(t) = \frac1{\sqrt N}{\Gamma^+_{-N+\lfloor 2 Nt \rfloor} },\qquad
\hat\Gamma^-(t) = \frac1{\sqrt N}{\Gamma^-_{-N+\lfloor 2 N t \rfloor}}.
\end{equation}
We are now ready to state the main result of this work.
\begin{theorem}\label{thm:Main}
	Fix \({\beta <\betac(q)}\). Then, for any \(\epsilon>0\),
	\begin{equation}\label{eq:Gam-pm}
	\lim_{N\to\infty} \FKlaw_{\beta,\Lambda_+}\bigl(\sup_{t\in[0,1]} \babs{ \hat\Gamma^+(t) - \hat\Gamma^-(t) } > \epsilon \bgiven v_L\leftrightarrow v_R \bigr) = 0 . 
	\end{equation}
	Furthermore, under the family of measures \(\{\FKlaw_{\beta,\Lambda_+}(\,\cdot \given v_L\leftrightarrow v_R)\}\), the following weak convergence result holds as $N\to\infty$:
	\begin{equation}\label{eq:wGamma}
	\hat\Gamma^+ \Rightarrow \sqrt{\chi}\excursion,
	\end{equation}
	where \(\mathfrak{e}:[0,1]\to\bbR\) is the normalized Brownian excursion and, as before, $\chi=\chi(\beta,q)$ is the curvature of the equilibrium crystal shape $\partial{\mathbf K}_{\beta^*}$ in the horizontal direction.
\end{theorem}

\subsection{Results in related settings}

We describe here a few results that would follow by minor adaptations of our analysis. We state the results in the language of high-temperature random-cluster measures, but there are straightforward reformulations in terms of the low-temperature Potts models. Let \(\Lambda_N = \{-N,\dots, N\}^2\) and let \(v_L,v_R\) and \(\Gamma=C_{v_L,v_R}\) be as before. Let \(\mathcal{L}\) be the set of edges with both endpoints having second coordinate \(0\). Define \(\FKlaw_{\beta,J,J',\Lambda}\) the random-cluster measure with edge weights \(e^{\beta}-1\) in \(\Lambda_+\setminus \mathcal{L}\), \(e^{J\beta}-1\) for edges in \(\mathcal{L}\) and \(e^{\beta J'}-1\) for edges having at least one endpoint in \(\{-N,\dots,N\}\times \{-1,\dots,-N\}\). In particular, \(\FKlaw_{\beta,1,0,\Lambda} = \FKlaw_{\beta,\Lambda_+}\) and the case \(J'=1\) is the defect line setting of~\cite{Ott+Velenik-2018}. Let  \(\Gamma^+\), \(\Gamma^-\), \(\hat{\Gamma}^+\) and \(\hat{\Gamma}^-\) be defined as before.
\begin{theorem}\label{thm:Main_Weak_couplings}
	Fix \(\beta<\betac(q)\), \(0\leq J'<1\) and \(0\leq J\leq 1\).
	Then, for any \(\epsilon>0\),
	\[
	\lim_{N\to\infty} \FKlaw_{\beta,J,J',\Lambda}\bigl(\sup_{t\in[0,1]} \babs{ \hat\Gamma^+(t) - \hat\Gamma^-(t) } > \epsilon \bgiven v_L\leftrightarrow v_R \bigr) = 0
	\]
	and
	\[
	\hat\Gamma^+ \Rightarrow \sqrt{\chi}\excursion,
	\]
	where \(\chi\) and \(\excursion\) are as in Theorem~\ref{thm:Main}.
\end{theorem}

\begin{theorem}\label{thm:Main_Line}
	Fix \(\beta<\betac(q)\) and \(0\leq J< 1\).
	Then, for any \(\epsilon>0\),
	\[
	\lim_{N\to\infty} \FKlaw_{\beta,J,1,\Lambda}\bigl(\sup_{t\in[0,1]} \babs{ \hat\Gamma^+(t) - \hat\Gamma^-(t) } > \epsilon \bgiven v_L\leftrightarrow v_R \bigr) = 0
	\]
	and
	\[
	\Psi \Rightarrow \frac{1}{2}\nu^{+ } + \frac{1}{2}\nu^{- }
	\]
	where \(\Psi\) is the law of \(\hat\Gamma^+\) and \(\nu^{\pm }\) are the law of \( \pm\sqrt{\chi}\excursion\), and the rest is as in Theorem~\ref{thm:Main}.
\end{theorem}
Finally, the results and techniques developed in Sections~\ref{sec:Entropic_repulsion}--\ref{sec:Fluct} pave the way for proving the following statement (the rather tedious details are omitted; see~\cite{Campanino+Ioffe+Louidor-2010} for the proof of a similar statement):
\begin{theorem}
	\label{thm:Connection}
	Fix \(\beta<\betac(q)\). For any pair \((J,J')\) satisfying \(0\leq J'<1\) and \(0\leq J\leq 1\) or \(J'=1\) and \(0\leq J<1\), there exists \(C\geq 0\) (depending on \(\beta,q,J,J'\)) such that
	\begin{equation*}
		\FKlaw_{\beta,J,J',\Lambda}(v_L\leftrightarrow v_R) = \frac{C}{N^{3/2}} e^{-2\tau N}(1+\sfo_N(1)).
	\end{equation*}
\end{theorem}

\subsection{Organization of the Paper}

In Section~\ref{sec:OZinfVol_Consequences} we present some results about the geometry of long connections in the infinite-volume random-cluster measure and deduce that typically, under \(\FKlaw_{\beta,\Lambda_+}(\,\cdot \given v_L\leftrightarrow v_R)\), the long cluster has the structure of a concatenation of small ``irreducible'' pieces. Section~\ref{sec:Entropic_repulsion} is devoted to the proof that the long cluster under \(\FKlaw_{\beta,\Lambda_+}(\,\cdot \given v_L\leftrightarrow v_R)\) is repulsed far away from the lower boundary of \(\Lambda_+\). We use this repulsion result in Section~\ref{sec:Effective_HS_RW} to construct a coupling between \(\Gamma\) under \(\FKlaw_{\beta,\Lambda_+}(\,\cdot \given v_L\leftrightarrow v_R)\) and an effective {semi}-directed random walk conditioned to stay in the upper half-plane. The latter is studied in Section~\ref{sec:Fluct} where an invariance principle to Brownian excursion is proven for a general class of such {semi}-directed random walks.

\section{Diamond Decomposition and Ornstein--Zernike Theory}
\label{sec:OZinfVol_Consequences}

The main result we will need to import is the Ornstein--Zernike representation of long subcritical clusters derived in~\cite{Campanino+Ioffe+Velenik-2008} and~\cite{Ott+Velenik-2018}. A random-walk representation of long subcritical clusters under {the unique infinite-volume measure}  {\(\FKlaw\)} 
 was constructed in~\cite{Campanino+Ioffe+Velenik-2008} in the general framework of Ruelle transfer operator for full shifts. In~\cite[Section 4]{Ott+Velenik-2018} 
an improved renewal version of~\cite{Campanino+Ioffe+Velenik-2008} was developed. We recall here the main objects and the result we will use.

\subsection{Cones and Diamonds}

We first define the cones and the associated diamonds:
\begin{gather*}
	\fcone = \setof{i\in\Z^2}{i_1 \geq \abs{i_2}}, \quad \bcone = -\fcone,\\
	\diam(u,v) = (u+\fcone)\cap (v+\bcone).
\end{gather*}
We will also need, for \(\delta>0\), \(\fcone_{\delta} = \setof{i\in\Z^2}{\delta i_1 \geq \abs{i_2}}\). Of course, \(\fcone=\fcone_1\).

Let \(\gamma=(V_{\gamma},E_{\gamma})\) be a connected subgraph of \(\Z^2\). We will say that \(\gamma\) is:
\begin{itemize}
	\item \emph{Forward-confined} if there exists \(u\in V_{\gamma}\) such that \(V_{\gamma}\subset u+\fcone\). When it exists, such a \(u\) is unique; we denote it \(\fend({\gamma})\).
	\item \emph{Backward-confined} if there exists \(v\in V_{\gamma}\) such that \(V_{\gamma}\subset v+\bcone\). When it exists, such a \(v\) is unique; we denote it \(\bend({\gamma})\).
	\item \emph{Diamond-confined} if it is both forward- and backward-confined.
	\item \emph{Irreducible} if it is diamond-confined and it is not the concatenation of two other diamond-confined graphs (see below for the definition of concatenation).
\end{itemize}

We will say that \(v\in \gamma\) is a \emph{cone-point} of \(\gamma\) if
\begin{equation*}
	V_{\gamma}\subset v+ (\bcone\cup \fcone).
\end{equation*}
We denote \( \CPts(\gamma)\) the set of cone-points of \(\gamma\).

We call a graph with a distinguished vertex a \emph{marked graph}. The distinguished vertex is denoted \(v^*\). Define
\begin{itemize}
	\item The sets of confined pieces:
	\begin{gather*}
	\SetRootMarkBackCont = \{\gamma \text{ marked backward-confined with } v^* =0 \},\\
	\SetRootMarkForwCont = \{\gamma \text{ marked forward-confined with } \fend(\gamma) =0 \},\\
	\SetRootDiaCont = \{\gamma \text{ diamond-confined with } \fend(\gamma) =0 \},\\
	\SetRootDiaCont^{\textnormal{irr}} = \{\gamma \text{ irreducible with } \fend(\gamma)=0\}.
	\end{gather*}
	We see {that} \(\SetRootDiaCont\) could be viewed as a subset of both \(\SetRootMarkBackCont\) (via the marking of \(\fend(\gamma)\)) and \(\SetRootMarkForwCont\) (via the marking of \(\bend(\gamma)\)).
	To fix ideas we shall, unless stated otherwise, think of $\SetRootDiaCont$ as of a subset of $\SetRootMarkBackCont$, that is, by default the vertex $\fend (\gamma) = 0$ is marked for any 		$\gamma\in \SetRootDiaCont$. 
	\item The displacement along a piece:	
	\begin{equation}
	\label{eq:displacement}
		\displace(\gamma) \defby (\theta(\gamma), \zeta(\gamma)) =
		\begin{cases}
			\bend(\gamma)	& \text{ if } \gamma\in
			\SetRootMarkBackCont,\ \text{in particular, if $\gamma\in \SetRootDiaCont$,} \\
			v^* 			& \text{ if } \gamma\in\SetRootMarkForwCont.
		\end{cases}
	\end{equation}
	\item The \emph{concatenation} operation: for \(\gamma_1\in \SetRootMarkBackCont\) and \(\gamma_2\in \SetRootMarkForwCont\) define the concatenation of $\gamma_2$ to $\gamma_1$ as
	\begin{equation*}
		\gamma_1\concatenate \gamma_2 = \gamma_1\cup (\displace(\gamma_1) + \gamma_2).
	\end{equation*}
	The concatenation of two graphs in \(\SetRootDiaCont\) is an element of \(\SetRootDiaCont\) and the concatenation of
	 a graph in \(\SetRootDiaCont\) to an element of \(\SetRootMarkBackCont\) is an element of \(\SetRootMarkBackCont\). The displacement along a concatenation is the sum of the displacements along the pieces.
\end{itemize}

\subsection{Ornstein--Zernike Theory for long Clusters in Infinite Volume}

Recall that $\tau\eone \in \partial{\mathbf K}_{\beta^*}$ is the rightmost point on the boundary of the Wulff shape. It can be informally thought of as the proper drift to stretch phase separation lines in the horizontal direction, see the developments of the Ornstein--Zernike theory in~\cite{Ioffe-1998, Campanino+Ioffe-2002,Campanino+Ioffe+Velenik-2003,Campanino+Ioffe+Velenik-2008,Ioffe-2015,Ott+Velenik-2018}. 
The main claim we import from~\cite{Ott+Velenik-2018} is
\begin{theorem}
	\label{thm:OZ_main}
	There exist \(C\geq 0,c>0,\delta>0\) such that one can construct two positive finite measures  \(\rho_L,\rho_R\) on \(\SetRootMarkBackCont\) and \(\SetRootMarkForwCont\) and a probability measure \({\mathbf p}\) on \(\SetRootDiaCont\) such that, for any point \(x{= (x_1 , x_2 )}\in\fcone_{\delta}\) and any bounded function \(f\) of the cluster of \(0\),
	\begin{multline*}
		\Babs{ e^{\tau \eone\cdot x }{\FKlaw} \bigl( f(C_{0,x})\IF{0\leftrightarrow x} \bigr) - \\ -\sum_{\gamma_L,\gamma_R} \rho_L(\gamma_L)\rho_R(\gamma_R)\sum_{M\geq 0}\sum_{\gamma_1,\dots,\gamma_M} \IF{\displace(\gamma)=x}f(\gamma) \prod_{i=1}^M \bfp(\gamma_i) } \leq C\norm{f}_{\infty}e^{-c\norm{x}} ,
	\end{multline*}
	where the sums are over \(\gamma_L\in\SetRootMarkBackCont,\gamma_R\in\SetRootMarkForwCont\) and \(\gamma_i\in\SetRootDiaCont\), 
	{such that the displacement along the concatenation } \(\gamma=\gamma_L\concatenate \gamma_1\concatenate\cdots \concatenate \gamma_M\concatenate \gamma_R\) 
	{satisfies  $\displace  (\gamma )= x$}. Moreover, there exist \(C'\geq 0,c' >0\) such that
	\begin{equation}
	\label{eq:OZ_expDec_steps}
	{
		\max \lbr
		\rho_L(\norm{\displace(\gamma_L)}\geq l) , \rho_R(\norm{\displace(\gamma_R)}\geq l) , {\mathbf p}(\norm{\displace(\gamma_1)}\geq l)\rbr  \leq C'e^{-c' l}.
	}
	\end{equation}
\end{theorem}

\begin{remark}
	\label{rem:FullSpace_CPts_density}
	In particular, Theorem~\ref{thm:OZ_main} implies that, up to exponentially small error, \(C_{0,x}\) has a linear (in \(\norm{x}\)) number of cone-points under \(\FKlaw(\,\cdot \given 0\leftrightarrow x )\).
\end{remark}

\subsection{Cone-Points {of} the Half-Space Clusters}

We make here our first use of Theorem~\ref{thm:OZ_main}.
\begin{lemma}
	\label{lem:CPTs}
	Denoting \(\Gamma = C_{v_L,v_R}\). There exist \(\rho>0\) and \(c>0\) such that
	\begin{equation}
	\label{eq:HalfSpace_CPts_density}
	{	\FKlaw_{\Lambda_+}} \bigl( \abs{\CPts(\Gamma)} \leq \rho N \bgiven v_L\leftrightarrow v_R \bigr) \leq e^{-cN}.
	\end{equation}
	Moreover, there exist \(c>0,C\geq 0\) such that
	\begin{equation}
	\label{eq:HalfSpace_CPts_spacing}
	{\FKlaw_{\Lambda_+}}\bigl( \max_{u,v\in \CPts(\Gamma)} \IF{\CPts(\Gamma)\cap ((u_1,v_1)\times \Z) = \varnothing}|u_1-v_1| \geq \log(N)^2 \bgiven v_L\leftrightarrow v_R\bigr) \leq \frac{C}{N^{c\log(N)}}.
	\end{equation}
\end{lemma}
{Note that the event $\lbr \CPts(\Gamma)\cap ((u_1,v_1)\times \Z) = \varnothing\rbr$ above simply means that $v$ and $u$ are successive cone points. 
}
\begin{proof}
	By the FKG property of the random-cluster measures, as {\(\FKlaw \succcurlyeq \FKlaw_{\Lambda_+}\)}, one can monotonically couple them (for example using the coupling described in Appendix~\ref{app:coupling}). Denote this coupling \(\Psi\) and let \((\omega,\eta)\) be a random vector of law \(\Psi\) with \(\omega\geq \eta\). In particular, for any non-decreasing event \(A\) {such that}  \(\{\eta\in A\}\), all pivotal edges for \(A\) in \(\omega\) are also pivotal for \(A\) in \(\eta\). In the same fashion {if \( \eta \in \lbr v_L\leftrightarrow v_R\rbr\)}, then all  the cone-points of \(\Gamma(\omega)\) are also cone-points of \(\Gamma(\eta)\). Via Remark~\ref{rem:FullSpace_CPts_density}, Theorem~\ref{thm:OZ_main} implies that there exist \(\rho>0\) and \(c>0\) such that
	\begin{equation*}
		{\FKlaw}( \abs{\CPts(\Gamma)}\leq \rho N ,v_L\leftrightarrow v_R )\leq e^{-cN}e^{-2\tau N}.
	\end{equation*}
	Then, by monotonicity and the previous observation on the inclusion of pivotal edges,
	\begin{align*}
		\FKlaw_{\Lambda_+} \bigl( \abs{\CPts(\Gamma)} \leq \rho N, v_L\leftrightarrow v_R \bigr)\leq \FKlaw( \abs{\CPts(\Gamma)}\leq \rho N ,v_L\leftrightarrow v_R ),
	\end{align*}
	implying~\eqref{eq:HalfSpace_CPts_density} as \(\FKlaw_{\Lambda_+} (v_L\leftrightarrow v_R )=e^{-2\tau N(1+\sfo(1))}\). Indeed,
	\begin{align*}
			\FKlaw_{\Lambda_+} \bigl( \abs{\CPts(\Gamma)} \leq \rho N \bgiven v_L\leftrightarrow v_R \bigr)
			&=
			\frac{\FKlaw_{\Lambda_+} \bigl( \abs{\CPts(\Gamma)} \leq \rho N , v_L\leftrightarrow v_R \bigr)}{ \FKlaw_{\Lambda_+} \bigl( v_L\leftrightarrow v_R \bigr)}\\
			&\leq
			\frac{ e^{-cN}e^{-2\tau N} }{ e^{-2\tau N(1+\sfo(1))}} \leq e^{-cN/2},
	\end{align*}
	for \(N\) large enough.
	To get~\eqref{eq:HalfSpace_CPts_spacing}, let \(w_1,\dots,w_m\) be the first coordinate of the cone-points of \(\Gamma(\omega)\), ordered from left to right, and let \(l_i=w_{i+1}-w_i, i=1,\dots,m-1\). Denote by \(w_j'\), \(m'\) and \(l_j'\) the corresponding quantities for \(\Gamma(\eta)\). The left-hand side of~\eqref{eq:HalfSpace_CPts_spacing} becomes
	\begin{equation*}
		\frac{\Psi\bigl( \max_{j\in\{1,\dots,m'\}} l_j' \geq \log(N)^2 , v_L\xleftrightarrow{\eta} v_R\bigr)}{\FKlaw_{\Lambda_+} (v_L\leftrightarrow v_R )}.
	\end{equation*}
	Now, as the cone-points of \(\Gamma(\omega)\) are included in the cone-points of \(\Gamma(\eta)\), 
	\begin{equation*}
		\max_{j\in\{1,\dots,m'\}} l_j'\leq \max_{i\in\{1,\dots,m\}} l_i.
	\end{equation*}
    Notice that both \(\Gamma(\omega)\) and \(\Gamma(\eta)\) are well defined as \(\eta\in\{v_L\leftrightarrow v_R\}\) and \(\omega\geq \eta\). Using the lower bound \(\FKlaw_{\Lambda_+} (v_L\leftrightarrow v_R )\geq CN^{-3/2}e^{-2\tau N}\) from Lemma~\ref{lem:NonSharpLB} and the bound \(\FKlaw(v_L\leftrightarrow v_R)\leq e^{-2\tau N}\), one obtains
	\begin{align*}
		\frac{\Psi\Bigl( \displaystyle\max_{j\in\{1,\dots,m'\}} l_j' \geq \log(N)^2 , v_L\xleftrightarrow{\eta} v_R\Bigr)}{\FKlaw_{\Lambda_+} (v_L\leftrightarrow v_R )}
        &\leq \frac{\Psi\Bigl(\displaystyle \max_{i\in\{1,\dots,m\}} l_i \geq \log(N)^2 , v_L\xleftrightarrow{\omega} v_R\Bigr)}{CN^{-3/2}e^{-2\tau N}}\\
		&\leq
		C^{-1}N^{3/2}\FKlaw\bigl( \max_{i\in\{1,\dots,m\}} l_i \geq \log(N)^2 \bgiven v_L\leftrightarrow v_R\bigr).
	\end{align*}
    The bound in~\eqref{eq:HalfSpace_CPts_spacing} thus follows from~\eqref{eq:OZ_expDec_steps} and standard estimates on the maximum of an \iid family.
\end{proof}
For future use, it is convenient to reformulate Lemma~\ref{lem:CPTs} as follows: 
\begin{corollary} 
	\label{cor:IRD}
	There exist \(\rho>0\), \(C>0\) and \(c>0\) such that the following statements hold for all $N$ sufficiently large:
	\begin{enumerate}[label=\arabic*.]
	\item Up to an event of probability at most \(e^{-c N}\) under $\FKlaw_{\Lambda_+} ( \, \cdot \given  v_L\leftrightarrow v_R)$, the open cluster $C_{v_L, v_R}$ admits an irreducible decomposition
	 \begin{equation}\label{eq:IRD} 
		 C_{v_L, v_R} = \gamma_L\circ\gamma_1\circ\dots\circ \gamma_k\circ\gamma_R, 
	 \end{equation}
	 with $\gamma_L\in\frB_L, \gamma_R\in \frB_R$ and with at least $k\geq \rho N$ irreducible pieces $\gamma_1, \dots ,\gamma_k\in \SetRootDiaCont^{\textnormal{irr}}$.
	 \item Up to an event of probability at most $\frac{C}{N^{c\log(N)}}$ under $\FKlaw_{\Lambda_+}(\, \cdot \given v_L\leftrightarrow v_R)$, the irreducible pieces (viewed as connected subgraphs of the graph $\bbZ^2$) in the 
	 decomposition~\eqref{eq:IRD} satisfy: 
	 \begin{equation}\label{eq:diam-bound} 
		 \max\{ \operatorname{diam}(\gamma_L), \operatorname{diam}(\gamma_1), \dots, \operatorname{diam}(\gamma_k), \operatorname{diam}(\gamma_R)\} 
		 \leq (\log N)^2 , 
	 \end{equation}
	 where $\operatorname{diam}(A)$ is the Euclidean diameter of a set $A\subset\bbR^2$.
	\end{enumerate}
\end{corollary}

\section{Entropic Repulsion}
\label{sec:Entropic_repulsion}

\subsection{A Rough Upper Bound}

We will use the coupling constructed in Appendix~\ref{app:coupling}. As in Appendix~\ref{app:coupling}, let {\(\FKlaw_{a,\Lambda}\)} to denote the random-cluster measure with weight \(e^{\beta}-1\) on edges in \(\Lambda_+\) and weight \(a\) on edges with an endpoint in \(\Lambda_-\). We denote by \(\coupling\) the coupling between \(\FKlaw_{0,\Lambda}= \FKlaw_{\Lambda_+}\) and \(\FKlaw_{e^{\beta}-1,\Lambda}=\FKlaw_{\Lambda}\).
\begin{lemma}
	\label{lem:wall_to_energy_cost}
	For any \( \sfu , \sfv\in\Lambda_+\) and \(0\leq a < e^{\beta}-1\), 
	\begin{equation}
	\label{eq:wall_to_energy_cost}
		\FKlaw_{a,\Lambda_+}(\sfu\leftrightarrow \sfv) \leq \FKlaw\Bigl(\IF{\sfu\leftrightarrow \sfv} \bigl(1-\epsilon(a)\bigr)^{|\Piv(\sfu\leftrightarrow \sfv) \cap \Lambda_-|} \Bigr),
	\end{equation}
	where \(\FKlaw\) is the random-cluster measure on \(\Ztwo\) with edge weight \(e^{\beta}-1\) and \(\epsilon(a) = \frac{e^{\beta}-1-a}{(e^{\beta}-1+q)(e^{\beta}-1)}\).
\end{lemma}
\begin{proof}
	Let \((\omega,\eta)\sim \coupling\) be as in the Appendix (\(\omega\sim \FKlaw_{\Lambda}\)). Using the monotonicity of \(\coupling\),
	\begin{align}
	\label{eq:p-full-bound}
		\FKlaw_{a,\Lambda_+}(\sfu\leftrightarrow \sfv)
        &=
        \coupling\bigl(\eta\in \{\sfu\leftrightarrow \sfv\}\bigr)\\
        \nonumber
		&=
        \coupling\bigl(\eta,\omega\in \{\sfu\leftrightarrow \sfv\}\bigr)\\
        \nonumber
		&=
        \sum_{\substack{w\in \{0,1\}^{E_{\Lambda} }\\
        	 w\in\{\sfu\leftrightarrow \sfv\}}} \coupling\bigl(\omega = w, \eta\in \{\sfu\leftrightarrow \sfv\}\bigr)\\
         \nonumber
		&\leq
        \sum_{\substack{w\in \{0,1\}^{E_{\Lambda} }\\ w\in\{\sfu\leftrightarrow \sfv\}}} \coupling\bigl(\omega = w, \eta_e=1\ \forall e\in \Piv_{w}(\sfu\leftrightarrow \sfv)\cap \Lambda_- \bigr)\\
        \nonumber 
		&\leq
        \sum_{\substack{w\in \{0,1\}^{E_{\Lambda} }\\ w\in\{\sfu\leftrightarrow \sfv\}}} (1-\epsilon)^{|\Piv_w(\sfu\leftrightarrow \sfv)\cap \Lambda_-|}\FKlaw_{\Lambda}(\omega = w)\\
        \nonumber 
		&=
        \FKlaw_{\Lambda}\Bigl(\IF{\sfu\leftrightarrow \sfv} (1-\epsilon)^{|\Piv(\sfu\leftrightarrow \sfv)\cap \Lambda_-|}\Bigr).
	\end{align}
    The first inequality is inclusion of events and the second one is~\eqref{eq:strictMonotCoupling} with \(\epsilon = \epsilon(a) = \frac{e^{\beta}-1-a}{(e^{\beta}-1+q)(e^{\beta}-1)}\). Now, as \(1>\epsilon>0\), \(\IF{\sfu\leftrightarrow \sfv} (1-\epsilon)^{|\Piv(\sfu\leftrightarrow \sfv)\cap \Lambda_-|}\) is a nondecreasing function (opening an edge can only decrease the number of pivotal once the event is satisfied). Thus, monotonicity of random-cluster measure implies
	\begin{equation*}
		\FKlaw_{\Lambda}\Bigl(\IF{\sfu\leftrightarrow \sfv} (1-\epsilon)^{|\Piv(\sfu\leftrightarrow \sfv)\cap \Lambda_-|}\Bigr) \leq \FKlaw\Bigl(\IF{\sfu\leftrightarrow \sfv} (1-\epsilon)^{|\Piv(\sfu\leftrightarrow \sfv)\cap \Lambda_-|}\Bigr). \qedhere
	\end{equation*}
\end{proof}
\begin{remark}\label{rem:LineVsHalfSpace}
	In the case of the wall (\(a=0\)), one has the following simplification: since the function $\eta \mapsto 
	\IF{\sfu\leftrightarrow \sfv ;   \Piv(\sfu\leftrightarrow \sfv)\cap \Lambda_- = \emptyset} (\eta )$ is non-decreasing, one could have used instead
	\begin{align*}
	\FKlaw_{\Lambda_+}(u\leftrightarrow v) 
	&=
	\FKlaw_{\Lambda_+}(u\leftrightarrow v, \Piv(u\leftrightarrow v)\cap \Lambda_- = \emptyset) \\
	&\leq 
	\FKlaw_{\Lambda}\bigl(u\leftrightarrow v ,  \Piv(u\leftrightarrow v)\cap \Lambda_- = \emptyset \bigr).
	\end{align*}
	We will however work with~\eqref{eq:wall_to_energy_cost}, as we want to keep the proof straightforwardly adaptable to the case of Theorem~\ref{thm:Main_Weak_couplings}.
\end{remark}

\begin{lemma}
	\label{lem:HalfSpace_Connexions_UB}
	There exists \(c\geq 0\) such that, for any \(\sfu = (k , u), 
	\sfv = (k+m ,v) \in \Lambda_+\)	with \(m\) large enough and 
	\( u, v \leq \sqrt{m}\),
	\begin{equation}
		\label{eq:HalfSpace_Connexions_UB}
		e^{\tau m}\FKlaw_{\Lambda_+}(\sfu\leftrightarrow \sfv)\leq \frac{c(1+u)(1+v)}{m^{3/2}}.
	\end{equation}
\end{lemma}
The proof of Lemma~\ref{lem:HalfSpace_Connexions_UB} 
 relies on effective random walk estimates and it is relegated to Subsection~\ref{sub:le-proof}. 

\subsection{A Rough Lower Bound}

\begin{lemma}
	\label{lem:LB_HalfSpaceToFullSpace}
	For any \(u,v\in \Lambda_+\),
	\begin{equation*}
	\FKlaw_{\Lambda_+}( u \leftrightarrow v )
	\geq
    \FKlaw\bigl( C_{u} \subset \Lambda_+,\, u \leftrightarrow v \bigr).
	\end{equation*}
\end{lemma}
\begin{proof}
	\begin{align*}
	\FKlaw_{\Lambda_+}( u \leftrightarrow v )
    &=
    \sum_{\substack{ C\subset\Lambda_+ \\ C\ni u,v}} \FKlaw_{\Lambda_+}\bigl( C_{u}= C \bigr)\\
	&=
    \sum_{ C\ni u,v} \IF{C\subset \Lambda_+} \FKlaw_{C}\bigl( C\ \open \bigr)\FKlaw_{\Lambda_+}\bigl( \EdgeBnd C\ \close \bigr)\\
	&\geq
    \sum_{C\ni u,v} \IF{C\subset \Lambda_+} \FKlaw_{C}\bigl( C\ \open \bigr)\FKlaw\bigl( \EdgeBnd C\ \close \bigr)\\
	&=
    \FKlaw\bigl( C_{u} \subset \Lambda_+,\, u \leftrightarrow v \bigr),
	\end{align*}
	where the sums are over $C$ connected and the inequality is an application of FKG.
\end{proof}

From this inequality and Theorem~\ref{thm:OZ_main}, one can deduce the following
\begin{lemma}
	\label{lem:NonSharpLB}
	There exists a constant \(c>0\) such that, for all \(N>0\),
	\begin{equation}
	\label{eq:NonSharpLB}
	\FKlaw_{\Lambda_+}( v_L \leftrightarrow v_R ) \geq c\, N^{-3/2} e^{-2\tau N} .
	\end{equation}
\end{lemma}
The proof of Lemma~\ref{lem:NonSharpLB} also relies on effective random walk estimates and it is relegated to Subsection~\ref{sub:NonSharpLB-prf}. 

\subsection{Bootstrapping}

We start by proving a BK-type inequality for a certain type of events.
\begin{lemma}
	\label{lem:BK_like_inequality}
	Let \(G=(V_G,E_G)\) be a graph and let \(F=(V_F,E_F)\) be a finite subgraph of \(G\). Let \(\eta\in\{0,1\}^{E_G}\). Denote {\(\FKlaw_{F}^{\eta}\)} 
	the random-cluster measure on \(E_F\) with edge weight \(e^{\beta}-1\geq 0\), cluster weight \(q\geq 1\) and boundary condition \(\eta\). For \(u,v\in V_F\) and \(e\in E_F\), denote \(A_e(u,v)\) the event that there exists an open path from \(u\) to \(v\) not using \(e\).	
	Then, for any \(e=\{i,j\}\in E_F\) and any \(x,y\in V_F\),
	\begin{equation}
	\label{eq:BK_like_inequality}
	\FKlaw_F^{\eta}\bigl( A_e(x,i),A_e(j,y),\, \omega_e=1,\, e\in\Piv(x\leftrightarrow y) \bigr)
    \leq
    e^{\beta} \FKlaw_F^{\eta}\bigl(x\leftrightarrow i \bigr)\FKlaw_F^{\eta}\bigl( i\leftrightarrow y \bigr).
	\end{equation}
\end{lemma}
\begin{proof}
	First notice that
	\begin{multline}
	\label{eq:BK_eq1}
	\FKlaw_F^{\eta}\bigl( A_e(x,i),\, A_e(j,y),\, \omega_e=1,\, e\in\Piv(x\leftrightarrow y) \bigr)=\\
    =
    \frac{e^{\beta}-1}{q} \FKlaw_F^{\eta}\bigl( A_e(x,i),\, A_e(j,y),\, \omega_e=0,\, e\in\Piv(x\leftrightarrow y) \bigr).
	\end{multline}
	Summing over the possible realizations of the cluster of \(x\) and \(i\),
	\begin{align*}
	\FKlaw_F^{\eta}\bigl( A_e(x,i),&\, A_e(j,y),\, \omega_e=0,\, e\in\Piv(x\leftrightarrow y) \bigr)=\\
	&=
    \sum_{\substack{C\ni x,i,\, C\not\ni j,y \\ \EdgeBnd C\ni e}} \FKlaw_F^{\eta}\bigl( C \open,\, \EdgeBnd C \close\bigr) \FKlaw_F^{\eta}\bigl( j\leftrightarrow y \bgiven \EdgeBnd C \close\bigr)\\
	&\leq
    \FKlaw_F^{\eta}( j\leftrightarrow y ) \sum_{C\ni x,i} \FKlaw_F^{\eta}\bigl( C \open,\, \EdgeBnd C \close\bigr)\\
	&=
    \frac{\FKlaw_F^{\eta}( j\leftrightarrow y ) \FKlaw_F^{\eta}( \omega_e=1 )}{\FKlaw_F^{\eta}( \omega_e=1 )} \FKlaw_F^{\eta}( i\leftrightarrow x )
    \leq \frac{e^{\beta}-1+q}{e^{\beta}-1} \FKlaw_F^{\eta}( i\leftrightarrow y ) \FKlaw_F^{\eta}( i\leftrightarrow x ).
	\end{align*}
    The first inequality is FKG and the second is FKG and finite energy (that is, the fact that the probability for an edge to be open, conditionnally on all the other edges, is uniformly bounded away from \(0\) and \(1\)). Plugging this into~\eqref{eq:BK_eq1} yields the result.
\end{proof}

This Lemma will prove useful as cone-points events imply the events in the left-hand side of~\eqref{eq:BK_like_inequality}. First, by~\eqref{eq:HalfSpace_CPts_spacing} and the definition of \(\fcone\), we have
\begin{equation}
\label{eq:HaudorffDist_Cluster_to_CPts}
	\FKlaw_{\Lambda_+}\bigl( \HausdorffDist(C_{v_L,v_R},\, \CPts(C_{v_L,v_R}))\leq (\log N)^2 \bgiven v_L\leftrightarrow v_R \bigr) \xrightarrow{N\to\infty} 1,
\end{equation}
where \(\HausdorffDist\) denotes the Hausdorff distance.
Moreover, this convergence is super-polynomial (the error decays faster than any negative power of \(N\)). 

Let \(\epsilon>0\), define
\begin{gather}
\label{eq:Delta}
	\Delta\equiv\Delta(N,\epsilon) = [-N+2N^{8\epsilon}, N-2N^{8\epsilon}]\times [0,N^{\epsilon}],\\
	\nonumber
	\tilde{\Delta} \equiv \tilde{\Delta} (N,\epsilon) = [-N+N^{8\epsilon}, N-N^{8\epsilon}]\times [0,2N^{\epsilon}].
\end{gather}

\begin{figure}[h]
	\centering
	\includegraphics[scale=0.8]{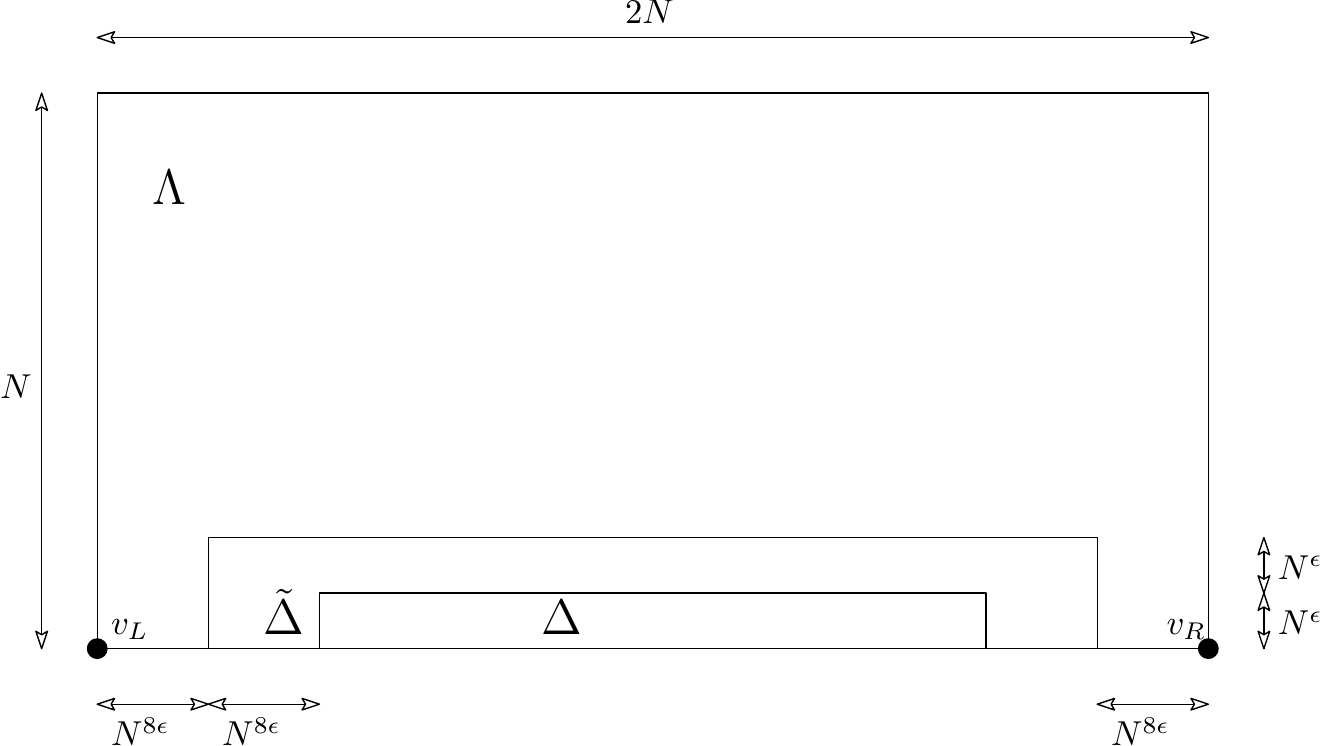}
\end{figure}

\begin{lemma}
	\label{lem:EntropicRepulsion}
	For any \(\epsilon\in(0,1/8) \), there exists \(C\geq 0\) such that
	\begin{equation}
	\label{eq:EntropicRepulsion}
		\FKlaw_{\Lambda_+} \bigl( C_{v_{ L},v_R}\cap \Delta \neq \varnothing \bgiven v_L\leftrightarrow v_R \bigr) \leq C N^{-\epsilon}.
	\end{equation}
\end{lemma}
\begin{proof}
	By~\eqref{eq:HaudorffDist_Cluster_to_CPts}, we can suppose that \(\HausdorffDist(C_{v_L,v_R}, \CPts(C_{v_L,v_R}))\leq (\log N)^2\). Under this event, \(\{C_{v_L,v_R}\cap \Delta\neq \varnothing \}\) implies \(\{\CPts(C_{v_L,v_R}) \cap \tilde{\Delta}\neq \varnothing\}\) (for \(N\) large enough). By a union bound, the probability of the latter is bounded from above by
	\begin{align}
	\label{eq:AboveDelta}
		\FKlaw_{\Lambda_+} \bigl( \CPts(C_{v_L,v_R}) \cap \tilde{\Delta}\neq \varnothing &\bgiven v_L\leftrightarrow v_R \bigr)\leq\\
		\nonumber
        &\leq
        \sum_{u\in\tilde{\Delta}} \frac{\FKlaw_{\Lambda_+} \bigl( u\in \CPts(C_{v_L,v_R}),\, v_L\leftrightarrow v_R \bigr)}{\FKlaw_{\Lambda_+} \bigl(v_L\leftrightarrow v_R \bigr)}\\
        \nonumber
		&\leq
        C e^{2\tau N}{N}^{3/2} \sum_{u\in\tilde{\Delta}} e^{\beta} \FKlaw_{\Lambda_+} \bigl( v_L\leftrightarrow u \bigr) \FKlaw_{\Lambda_+} \bigl( u\leftrightarrow v_R \bigr)\\
        \nonumber
		&\leq
        C {N}^{3/2} \sum_{k=N^{8\epsilon}}^{2N-N^{8\epsilon}} \sum_{l=0}^{2N^{\epsilon}} (1+l)^{2} k^{-3/2} (2N-k)^{-3/2}\\
        \nonumber
		&\leq
        C N^{3\epsilon} \sum_{k=N^{8\epsilon}}^{N/2} k^{-3/2}
        \leq
        C N^{3\epsilon} N^{- 4\epsilon}
        \xrightarrow{N\to\infty} 0,
	\end{align}
	where the first line follows from a union bound, the second one from~\eqref{eq:BK_like_inequality} {(since by construction if
		\( u\in \CPts(C_{v_L,v_R}) \),  then the bonds  $\langle u -\sfe_1  , u\rangle $ and  $\langle u , u+\sfe_1\rangle $ are  pivotal for $\lbr v_L \leftrightarrow v_R\rbr$) and Lemma~\ref{lem:NonSharpLB}},  and the third one from Lemma~\ref{lem:HalfSpace_Connexions_UB}. By convention the constant  \(C\) is updated at each line.
\end{proof}

\section{Proof of Theorems~\ref{thm:Main}, \ref{thm:Main_Weak_couplings} and~\ref{thm:Main_Line}}
\label{sec:Effective_HS_RW}

We focus on the proof of Theorem~\ref{thm:Main}. The necessary adaptations needed to prove the other two theorems are sketched in Section~\ref{sec:TheOtherTwoThms}.

Throughout this Section we fix $\epsilon \in (0, 1/16)$, which is used to
define the rectangle  $\Delta$ in \eqref{eq:Delta} and, subsequently, shows
up in the statement of the entropic repulsion Lemma~\ref{lem:EntropicRepulsion}. To facilitate notation we set $\delta = 8\epsilon \in (0, 1/2)$. 

\subsection{Reduction to infinite volume quantities}

Consider the irreducible decomposition~\eqref{eq:IRD}. In view of Corollary~\ref{cor:IRD}, we may restrict attention to clusters $C_{v_L , v_R}$ which contain cone-points in any vertical slab of width $(\log N)^2$.
In the sequel, we shall use $\calS_{a,b}$ for the vertical slab through the vertices $(a, 0)$ and $(b ,0)$.  

Let $\sfu_L$ be the left-most cone-point of $C_{v_L , v_R}$ in $\calS_{-N+2N^{\delta}, -N + 2N^\delta + (\log N )^2}$. Similarly, let $\sfu_R$ be the right-most cone-point of $C_{v_L , v_R}$ in $\calS_{N-2N^{\delta} - (\log N )^2 , N - 2N^\delta } $. 
We record $\sfu_L$ and $\sfu_R$ in their coordinate representation as 
\begin{equation}\label{eq:vert} 
	\sfu_L  = (j_L , u_L)\quad \text{ and }\quad \sfu_R = (j_R , u_R) .
\end{equation}
By construction, since $\sfu_L \in v_L +\fcone$ and $\sfu_R \in v_R +\bcone$, the vertical coordinates of $u_L$ and $u_R$ (see~\eqref{eq:vert}) satisfy 
\begin{equation}\label{eq:Heights} 
	u_L , u_R \leq \sqrt{2}\, \bigl( 2N^{\delta} + (\log N)^2 \bigr) .
\end{equation}
Gluing together all the irreducible pieces on the left of $\sfu_L$ and on the right of $\sfu_R$, we may modify~\eqref{eq:IRD} as follows: 
\begin{equation}\label{eq:IRD-points} 
	C_{v_L , v_R} = \eta_L \circ \eta_1\circ  \dots \circ \eta_k \circ \eta_R  = \eta_L\circ \ueta\circ \eta_R, 
\end{equation}
where $\eta_L = C_{v_L,v_R}\cap (\sfu_L + \bcone) \in \SetRootMarkBackCont$, $\eta_R = C_{v_L,v_R}\cap (\sfu_R + \fcone) \in 	\SetRootMarkForwCont$ and 
\begin{equation}\label{eq:eta-conc} 
\ueta\defby 
\eta_1\circ \dots \circ\eta_k = \gamma_{\ell+1}\circ \dots \circ\gamma_{\ell+k} = (\sfu_L + \fcone) \cap C_{v_L,v_R} \cap (\sfu_R +\bcone)  
\end{equation}
is the portion $\gamma_{\ell+1}\circ \dots \circ\gamma_{\ell+k}$ of the concatenation of all	$\SetRootDiaCont^{\textnormal{irr}}$-irreducible pieces located between $\sfu_L$ and $\sfu_R$ 
in the decomposition~\eqref{eq:IRD}. In~\eqref{eq:eta-conc}, we set  $\eta_j = \gamma_{\ell+j}$ for all $j = 1, \dots , k$.  

\begin{figure}[h]
	\label{fig:uLuR}
	\centering
	\includegraphics[scale=0.8]{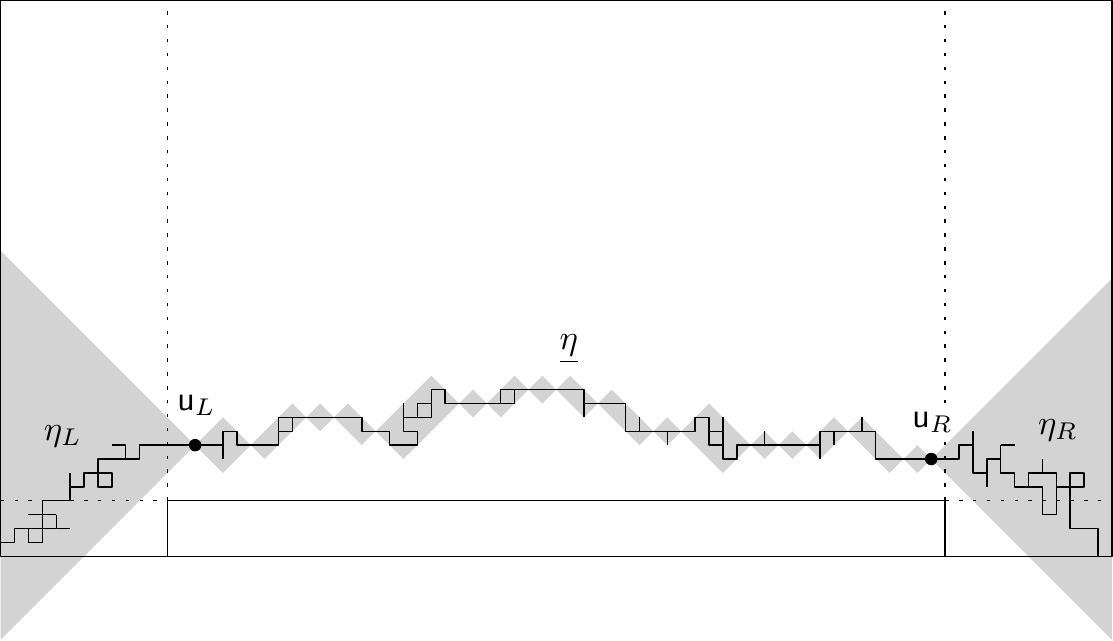}
	\caption{Decomposition of the cluster \(C_{v_L , v_R}\) as a concatenation \(\eta_L\circ \ueta\circ \eta_R\).}
	\label{fig:uLuR}
\end{figure}

By Lemma~\ref{lem:EntropicRepulsion}, we may restrict attention to the case when 
\begin{equation}\label{eq:NoInt} 
(\sfu_L + \ueta) \cap \Delta = \emptyset .
\end{equation}
In light of the above discussion, and with~\eqref{eq:IRD-points} and~\eqref{eq:eta-conc} in mind, 
it is natural do define the following set $\frT_N$: 
\begin{definition} 
\label{def:TN} 
We define \(\frT_N\) as the set of triples $(\eta_L , \eta_R , \ueta)$ (see Figure~\ref{fig:uLuR})
and the corresponding vertices (recall the definition of displacement in~\eqref{eq:displacement})
\[ 
\sfu_L \defby v_L + X(\eta_L),\quad \sfu_R \defby v_R - X(\eta_R) ,  
\] 
in their coordinate representation~\eqref{eq:vert}, such that 
\be{eq:etaH} 
v_L +\eta_L\circ \ueta\circ \eta_R \subset \bbH_+ \quad 
\text{and, furthermore,  \eqref{eq:NoInt} holds} .
\ee
Moreover, 
\begin{equation} 
\label{eq:TN} 
j_L \in [-N, -N +2N^{\delta} + (\log N )^2 ]\quad \text{ and }\quad j_R\in 
[N- 2N^\delta - (\log N )^2 , N] , 
\end{equation}
and $v_L +\eta_L$ and $\sfu_R +\eta_R$ do not have cone-points in the interior of the vertical slabs 
$\calS_{-N +2N^\delta, j_L}$ and $\calS_{j_R, N- 2N^\delta}$.
In addition, $\max_i \operatorname{diam} (\eta_i) \leq (\log N)^2$ and~\eqref{eq:Heights} holds.
\end{definition}
\begin{lemma} 
	\label{lem:C-structure} 
	There exist $c, C \in (0,\infty)$ such that, for all $N$ sufficiently large,
	\begin{equation}\label{eq:C-structure} 
		\FKlaw_{\Lambda_+} (v_L\leftrightarrow v_R) \bigl( 1 - CN^{-c \log N}\bigr)
		\leq 
		\sum_{(\eta_L, \eta_R, \ueta) \in \frT_N} \FKlaw_{\Lambda_+} (\eta_L\circ\eta_1\circ\dots\circ\eta_k\circ\eta_R) .
	\end{equation} 
\end{lemma} 
Now (see Section~3 in~\cite{Campanino+Ioffe+Velenik-2008}), the events in the right-hand side of~\eqref{eq:C-structure} can be represented as
\begin{equation}\label{eq:Events} 
\{\eta_L \circ \eta_1\circ  \dots \circ \eta_k \circ \eta_R\} = \{v_L + \eta_L\} \cap \{\sfu_L +\ueta\} \cap \{\sfu_R +\eta_R\}.
\end{equation} 
Thus, 
\begin{equation}\label{eq:RationEvents} 
\FKlaw_{\Lambda_+}(\eta_L \circ \eta_1\circ \dots \circ \eta_k \circ \eta_R) = \FKlaw_{\Lambda_+} (\sfu_L +\ueta \given v_L + \eta_L ; \sfu_R +\eta_R) \, \FKlaw_{\Lambda_+}(v_L + \eta_L ; \sfu_R +\eta_R) .
\end{equation}
In view of the sharpness of phase transition proved in~\cite{Duminil-Copin+Manolescu-2016}, the analysis of~\cite[Section~3]{Campanino+Ioffe+Velenik-2008} applies all the way up to the critical temperature. Consequently, by~(3.14) of the latter paper and the restriction~\eqref{eq:NoInt}, there exists $c \in (0, \infty)$ such that 
\begin{equation}\label{eq:VCondPhiBound}  
	\exp{- {\rm e}^{- c N^{\epsilon}}}
	\leq 
	\frac{\FKlaw_{\Lambda_+}(\sfu_L +\ueta \given v_L + \eta_L ; \sfu_R +\eta_R)}{\FKlaw(\sfu_L +\ueta \given v_L + \eta_L ; \sfu_R +\eta_R)}
	\leq
	\exp{ {\rm e}^{-c N^{\epsilon}}}
\end{equation}
for all $N$ sufficiently large, uniformly in $(\eta_L, \eta_R, \ueta) \in \frT_N$. 

Let us define the following regularized measure on $\frT_N$ or, equivalently, on the set of clusters
$C_{v_L, v_R} = \eta_R\circ\ueta\circ\eta_R$ with $(\eta_L, \eta_R, \ueta) \in \frT_N$:
\begin{equation}\label{eq:regPhi} 
	\FKlaw_{\Lambda_+}^{\mathsf{reg}} (\eta_L\circ\ueta\circ \eta_R)
	= 
	\frac{1}{Z_N}\FKlaw (\sfu_L +\ueta \given v_L + \eta_L ; \sfu_R +\eta_R) \, \FKlaw_{\Lambda_+} (v_L + \eta_L ; \sfu_R +\eta_R) ,
\end{equation}
where $Z_N = Z_N(\beta, \epsilon)$ is a normalizing constant.
We have proven
\begin{proposition}
	\label{prop:PhiPsi}
	There exists a coupling $\Psi_N$ between $\FKlaw_{\Lambda_+} (\, \cdot \given v_L\longleftrightarrow v_R)$ 
	(viewed as a probability distribution on the set of clusters $C_{v_L, v_R}$) and the probability distribution $\FKlaw_{\Lambda_+}^{\mathsf{reg}}$ on $\frT_N$ such that, for all $N$ sufficiently large,
	\begin{equation}\label{eq:PhiPsi} 
	\Psi_N\lb C_{v_L , v_R} \neq \eta_L\circ\ueta\circ\eta_R\rb 
	\leq 2C N^{-c\log N} .
	\end{equation}
\end{proposition}
From now on, we work only with the regularized measure $\FKlaw_{\Lambda_+}^{\mathsf{reg}}$. 

\subsection{Construction of the effective random walk} 
\label{sub:ERW}

Recall from~\eqref{eq:Delta} the definition of the rectangles $\Delta = \Delta(N,\epsilon)$.  
Let us, first of all, define a modified set of triples $\frT^*_N = (\lambda_L, \ulambda, \lambda_R)$
such that $\lambda_L\in\frB_L, \lambda_R\in\frB_R$ and, in addition, 
\begin{gather*} 
	\ulambda = \lambda_1\circ\dots\circ\lambda_M \text{ is a concatenation of $\lambda_i\in \frA$} \\
\intertext{and}
	v_L + \lambda_L\circ\ulambda\circ\lambda_R\subset \bbH_+\ 
	\textrm{and}\ \lb v_L + \lambda_L\circ\ulambda\circ\lambda_R\rb \cap\Delta = \emptyset.   
\end{gather*}
Note that irreducibility of the $\lambda_i$-s is not required here, since randomly glueing irreducible pieces together is necessary to recover independence in~\eqref{eq:ERW-steps}~\cite{IOSV,Ott+Velenik-2018}.
\smallskip

For $(\lambda_L, \ulambda, \lambda_R)\in \frT^*_N$, set 
\begin{equation}\label{uv-star}
	\sfu^*_L \defby (j^*_L, u^*_L) = v_L + X(\lambda_L), \ 
	\sfu^*_R \defby (j^*_R, u_R^*) = v_R - X(\lambda_R) = \sfu^*_L + X(\ulambda) .
\end{equation}
 Given two  probability measures \(\rho_{L,+},\rho_{R,+}\) on \(\SetRootMarkBackCont\) and \(\SetRootMarkForwCont\), respectively, and a probability measure $\sfp$ on $\frA$, one can construct the induced probability distribution $\sfP_+^*$ on $\frT_N^*$:
\begin{equation}\label{eq:Plus} 
	\sfP_+^* (\lambda_L\circ\ulambda\circ\lambda_R) =
	\frac{1}{Z_N^*}	\rho_L(\lambda_L)\, \rho_R(\lambda_R) \prod_{i=1}^M \sfp(\lambda_i) .
 \end{equation}
The product term on the right-hand side of the last expression is interpreted as an effective random walk with \iid steps distributed according to
\begin{equation}\label{eq:ERW-steps} 
	\sfP (X = \sfx) = \sum_{\lambda \in \frA} \sfp(\lambda) \IF{X(\lambda) = \sfx} .
\end{equation}
As in the case of Theorem~\ref{thm:OZ_main}, the following statement may be imported from~\cite{Ott+Velenik-2018} and from entropic repulsion estimates for random walks.
\begin{theorem}	
	\label{thm:OZ_main-Psi} 
	Let \({\mathbf p}\) be the (infinite-volume) probability measure on \(\SetRootDiaCont\) as it appears in Theorem~\ref{thm:OZ_main}. 
	 There exist \(C\geq 0,c>0\) such that, for any $N$ large enough, one can construct two  probability  measures \(\rho_{L,+}\) and \(\rho_{R,+}\) on \(\SetRootMarkBackCont\) and \(\SetRootMarkForwCont\), respectively, such that
	 \begin{equation}
	 \label{eq:BoundSteps}
	 	\max \bigl\{
	 	\rho_{L,+}\bigl(\theta(\lambda_L)\notin	[2N^\delta, 2N^\delta +\ell]\bigr), 
	 	\rho_{R,+}\bigl(\theta(\lambda_R)\notin	[2N^\delta, 2N^\delta +\ell]\bigr)
	 	\bigr\}
	 	\leq Ce^{-c\ell}.
	 \end{equation} 
	 Furthermore, there exists a coupling $\Psi^*_N$ between $\sfP_+^*$ and $\FKlaw_{\Lambda_+}^{\mathsf{reg}}$ such that 
	\begin{equation}\label{eq:Psi-star} 
		\Psi_N^*(\eta_L\circ\ueta\circ\eta_R \neq \lambda_R\circ\ulambda \circ\lambda_R) \leq C N^{-c\log N} .
	\end{equation}
\end{theorem}

\subsection{Surface tension, geometry of Wulff shape and diffusivity constant  of the effective random walk}
\label{sub:chi}

We follow the conventions for notation introduced in Subsection~\ref{sub:st-Wsh}. 
It will be convenient to write down explicit relations between the diffusivity constant of the effective random walk with \iid steps 
$X = (\theta, \zeta) \in \fcone$, the surface tension $\tau_{\beta^*}$ of the underlying Potts model and the curvature $\chi$ at $(\tau,0)\in \partial\mathbf{K}_{\beta^*}$, the boundary of the corresponding Wulff shape ${\mathbf K}_{\beta^*}$.

We know that $X$ has exponential moments in a neighborhood of the origin. Define 
\[ 
	\bfG(r,h) \defby \sfE\bigl( {\rm e}^{-r\theta + h\zeta}\bigr) .
\] 
Then (see, e.g., Theorem~3.2 in~\cite{Ioffe-2015}), the local parametrization of the boundary $\partial\bfK_\beta$ in a small neighborhood of $(\tau_\beta,0)$ can be recorded as follows:
\begin{equation}\label{eq:Part-Kb}
	(\tau_\beta-r,h) \in \partial \bfK_{\beta^*} \iff \bfG(r,h) = 1. 
\end{equation}
In view of lattice symmetries, a second-order expansion immediately yields the following formula for the curvature $\chi$: 
\begin{equation}\label{eq:chib-form} 
	\chi = \frac{\operatorname{\mathsf{Var}}({\zeta})}{\sfE(\theta)},  
\end{equation}
which coincides with the expression~\eqref{eq:chib-form-rw} for the diffusivity constant of the effective random walk. 

\subsection{Proof of Theorem~\ref{thm:Main}} 
\label{sub:ProofMain} 

In view of Proposition~\ref{prop:PhiPsi} and Theorem~\ref{thm:OZ_main-Psi}, it suffices to prove the invariance principle for the rescaling~\eqref{eq:GammaScaled} of the cluster $\Gamma = \lambda_L\circ\ulambda\circ\lambda_R$ under $\sfP_+^*$. Following~\eqref{eq:recaled-n}, let us define
\begin{equation}\label{eq:en-new}
	\hat{\fre}_N = \hat{\fre}_N(\lambda_L\circ\ulambda\circ\lambda_R) \defby
	\frI_N \Bigl( v_L, \sfu^*_L, \sfu^*_L + X(\lambda_1), \dots, \sfu^*_L + \sum_{i=1}^M X(\lambda_i) = \sfu^*_R, v_R \Bigr) .
\end{equation}
By Proposition~\ref{prop:PhiPsi} and Theorem~\ref{thm:OZ_main-Psi}, we may restrict attention to the case when the rescaled upper and lower envelopes $\Gamma^\pm$ defined in~\eqref{eq:Gam-pm} are close to $\hat{\fre}_N$ in the Hausdorff distance ${\rm d}_{\rm H}$ on $\bbR^2$, 
\begin{equation}\label{eq:HD}
	{\rm d}_{\text{H}}(\Gamma^\pm, \hat{\fre}_N) \leq \frac{(\log N)^2}{\sqrt{N}}, 
\end{equation}
which already implies~\eqref{eq:Gam-pm}. 
Therefore, it is enough to prove an invariance principle for $\hat{\fre}_N$ under $\sfP^*_+$. This, however, readily follows from Theorem~\ref{thm:conv} applied to the rescaling of middle pieces $\ulambda$ and our choice of $\delta = 8\epsilon <1/2$, which ensures that the rescaled boundary pieces $\lambda_L$ and $\lambda_R$ do not play a role. 
 
\subsection{Proofs of Theorems~\ref{thm:Main_Weak_couplings} and~\ref{thm:Main_Line}}
\label{sec:TheOtherTwoThms}
Theorem~\ref{thm:Main_Weak_couplings} is proved by the same argument as Theorem~\ref{thm:Main} (remember Remark~\ref{rem:LineVsHalfSpace}).
Theorem~\ref{thm:Main_Line} needs mostly the following adaptation: Lemma~\ref{lem:wall_to_energy_cost} will give a penalty whenever a cone-point is created on \(\mathcal{L}\) and not on the whole lower space. The same strategy used in the proof then shows that the cluster avoids the symmetrized version of \(\Delta(N,\epsilon)\) (see~\eqref{eq:Delta} in Section~\ref{sec:Entropic_repulsion}) with probability tending to one as \(N\to\infty\). Conditioning on the half-space containing the maximum of \(\Gamma^+\), one can then carry on the rest of the analysis and obtain Theorem~\ref{thm:Main_Line}.
 
\section{Fluctuation theory of the effective random walk}
\label{sec:Fluct}

\subsection{Effective random walk}
\label{sub:erw}

Theorem~\ref{thm:OZ_main}  and, subsequently, Theorem~\ref{thm:OZ_main-Psi} set up the stage for considering effective random walks \(\walk\) with $\bbN \times \bbZ$-valued \iid\ steps $X_1, X_2, \dots$, whose coordinates
will be denoted  as $X = (\theta , \zeta)$, and  which have the following set of properties: 
\begin{enumerate} 
	\item They have exponential tails: There exists $\alpha >0$ 
	such that $\bbE({\rm e}^{\alpha (\theta  +\abs{\zeta})}) <\infty$.  
	\item The conditional distribution ${\mathbf P}(\cdot \given \theta)$ of $\zeta$ is ${\mathbf P}$-a.s. symmetric, in particular  $\theta$ and $\zeta$ are uncorrelated. 
	\end{enumerate}
By Theorem~\ref{thm:OZ_main-Psi}, the displacements (recall \eqref{eq:displacement}) along 
diamond-confined clusters $\gamma \in \frA$ under $\sfp$, that is, 
\begin{equation}\label{eq:Setp-p} 
{\mathbf P}(\displace_i=x) = \sfp \lb \gamma\in \frA~:~\displace(\gamma) = x\rb  , 
\end{equation} 
satisfy the above assumptions.  

Define the diffusivity constant (compare with~\eqref{eq:chib-form})
\begin{equation}\label{eq:chib-form-rw} 
\chi = \frac{ \mathbf{Var} ({\zeta })}{\bfE ( \theta )}. 
\end{equation}
For $\sfu = (k, u)$, we use ${\mathbf P}_\sfu$
for the  random walk $\walk$ which starts at $\sfu$; $\walk_0 = \sfu$. Under ${\mathbf P}_\sfu$, the 
position $\walk_i$ of the walk after $i$ steps is given by 
\begin{equation}\label{eq:RW-steps} 
\walk_i = \sfu + \sum_{\ell =1}^i X_\ell \defby \sfu + \lb \sfT_i , \sfZ_i\rb, \ 
\text{where $\sfT_i = \sum_1^i\theta_\ell$ and $\sfZ_i = \sum_1^i \zeta_\ell$}.
\end{equation}
Given a subset  $A \subseteq   \bbN\times \bbZ$, or more generally $A\subset \bbR^2$,  define the hitting times 
\[ 
H_A = \inf\lbr i~:~ \walk_i \in A\rbr\ \text{and write $H_v = 
	H_{\lbr v\rbr}$ for vertices.} 
\]
Furthermore, given a subset $A\subseteq   \bbN\times \bbZ$ and a stopping time $H$ write 
\[ 
\calL_A ( H ) = \# \lbr i\leq H ~:~ \walk_i\in A\rbr .
\]
for the local time of $\walk$ at $A$ during the time interval $[0, H]$.

\subsection{Uniform repulsion estimates} 
\label{sub:Rep-ub}
We start with some general considerations and notation: 
 Let $U_n$ be a zero mean one-dimensional  random walk with \iid increments $\xi_k$. A function $h$ is called harmonic for $U_n$ killed at leaving the positive half-line if it solves the equation
\[
h(x)=\mathbf{E}[h(x+\xi_1);x+\xi_1>0],\quad x\ge0.
\]
According to Doney \cite{Doney98}, every positive solution to this equation is a multiple of the renewal function based on ascending ladder heights.
If one assumes that the increments $\xi_k$ have finite variance then ladder heights have finite expectations. Therefore, by the standard renewal theorem,
the corresponding renewal function is asymptotically linear. As a result,
\[
h(x)\sim Cx\quad\text{as }x\to\infty.
\]
In what follows, we will choose harmonic functions for which the latter relation holds with $C=1$. For this choice of the constant one has the representation
\[
h(x)=x-\mathbf{E}_x[U_\tau],
\]
where, with a slight abuse of notation we used $\mathbf{E}_x$ for the expectation with respect to the one-dimensional random walk $U_n$, which starts at $x \in \bbZ$, and where 
\[
\tau \defby \inf\{n\ge1: U_n\le0\}.
\]
Furthermore, $\mathbf{E}_x[U_\tau]$ converges, as $x\to\infty$, to a constant.

Let us go back 
to our $\bbN \times \bbZ$-valued effective random walks \(\walk\) as described in  Subsection~\ref{sub:erw}. 
Set  $\bbH_-$ to be  the lower half-plane, 
	\begin{equation}\label{eq:Hminus} 
	\bbH_- = \lbr x = (x_1, x_2 )~:~ x_2 < 0\rbr .
	\end{equation}	
First of all, the following asymptotic formula holds: 
\begin{theorem} 
	\label{thm:tpf}
	There exists a constant $C \in \bbR_+$ such that, as $n\to\infty$, 
\begin{equation}
\label{eq:tpf}
\mathbf{P}_{(0 , u)}\left( H_{(n , v)} <  
H_{\mathbb H_-} <\infty \rb 
\sim C\frac{h^+(u)h^-(v)}{n^{3/2}}
\end{equation}
uniformly  in $u,v \in ( 0, \delta_n\sqrt{n})\cap \bbN$, where
$\delta_n\to0$ arbitrarily slowly, and $h^\pm$ are positive harmonic functions
for random walks $\pm\sfZ_n$ killed when leaving the positive half-line.

Furthermore, there exists a constant $C$ such that 
\begin{equation}
\label{eq:tpf-a}
\mathbf{P}_{(0 , u)}\left( H_{(n , v)} <  
H_{\mathbb H_-} <\infty \rb 
\sim C\frac{h^+(u)ve^{-v^2/2n\bf{Var}(\zeta)}}{n^{3/2}}
\end{equation}
uniformly  in $u\in ( 0, \delta_n\sqrt{n})\cap \bbN$
and $v \in (\delta_n\sqrt{n},\sqrt{n})\cap \bbN$.

Finally, if $\delta_n\to0$ sufficiently slowly, then there exists a positive bounded function $\psi$ such that
\begin{equation}
\label{eq:tpf-b}
\mathbf{P}_{(0 , u)}\left( H_{(n , v)} <  
H_{\mathbb H_-} <\infty \rb 
\sim \frac{\psi(u/\sqrt{n},v/\sqrt{n})}{n^{1/2}}
\end{equation}
uniformly  in $u,v \in (\delta_n\sqrt{n},\sqrt{n})\cap \bbN$.
\end{theorem}
Note that, since the $\sfZ$-component of  $\walk$ is symmetric, 
$\mathbf{P}_{(0 , u)} (H_{\mathbb H_-} <\infty )=1$ for any $u\in \bbN$, and hence 
the events $\{ H_{\mathbb H_-} <\infty\}$ in~\eqref{eq:tpf}--\eqref{eq:tpf-b} are redundant. 
The statement of Lemma~\ref{lem:HalfSpace_Connexions_UB} relies  on the following fact, which is an analog 
of~\eqref{eq:tpf} for soft-core potentials: 
\begin{theorem} 
	\label{thm:eb-walk}  
	For any $\delta < 1$, there exists a constant $C$ such that 
	\begin{equation}\label{eq:eb-walk} 
	{\mathbf E}_{(0,u)}\bigl( \IF{H_{(n,v)} <\infty}\delta^{\calL_{\bbH_-}(H_{(n , v)})}\bigr) \leq \frac{ Cu v }{n^{3/2}} , 
	\end{equation}
	uniformly in $n$ and in $u, v \in (0, \sqrt{n})\cap \bbN$. 
\end{theorem}  

\subsection{Proof of Lemma~\ref{lem:HalfSpace_Connexions_UB}}
\label{sub:le-proof}

First, use~\eqref{eq:wall_to_energy_cost} and the fact that edges which are incident to cone-points are necessarily pivotal, to obtain
\begin{equation}
\label{eq:UB_eq1}
\FKlaw_{\Lambda_+}(\sfu\leftrightarrow \sfv) \leq \FKlaw\Bigl(\IF{\sfu\leftrightarrow \sfv} (1-\epsilon)^{|\CPts(C_{\sfu,\sfv})\cap \Lambda_-|} \Bigr).
\end{equation}
We proceed by deriving an upper bound on the right-hand side of~\eqref{eq:UB_eq1}, as a direct consequence of Theorem~\ref{thm:OZ_main} and of the random-walk estimate~\eqref{eq:eb-walk} of Theorem~\ref{thm:eb-walk}.
Let us denote $\sfu = (k , u)$ and $\sfv = (k +m ,v)$ with $0\leq  u, v\leq \sqrt{m}$.
Then, $\sfv\in \sfu + \fcone_{\delta}$ for all $m$ large and Theorem~\ref{thm:OZ_main} indeed applies, including the exponential bounds~\eqref{eq:OZ_expDec_steps}.
In particular, as far as the derivation of~\eqref{eq:HalfSpace_Connexions_UB} is concerned, we may restrict attention to boundary pieces $\gamma_L , \gamma_R$ satisfying $\| X (\gamma_L)\| , \| X (\gamma_L)\| \leq \lb\log m\rb^2$. Similarly, we may restrict attention to the case when the cluster $C_{\sfu , \sfv}$ does not go below $-N$.  

Let $\sfS$ be the random walk with  step distribution \({\mathbf p}\) defined in Theorem~\ref{thm:OZ_main}.  Due to the discussion in the preceding paragraph, we need to derive an upper bound on the restricted sum which can be recorded in the language employed in Subsection~\ref{sub:Rep-ub} as 
\begin{equation}\label{eq:RestrSum} 
\sum_{\|\sfx \| , \|\sfy \|\leq (\log m )^2} 
\rho_L\lb  X = \sfx\rb\rho_R ( X = \sfy )
{\mathbf E}_{\sfu+\sfx }\bigl( \IF{H_{\sfv - \sfy } <\infty}\delta^{\calL_{\bbH_-}(H_{\sfv - \sfy })}\bigr) .
\end{equation} 
Set $\sfw = \sfu + \sfx = (j, w)$ and $\sfz = \sfv -\sfy = ( j+n , z)$.
By construction, $n\in [m - (\log m)^2 , m]$. Applying~\eqref{eq:OZ_expDec_steps} and Theorem~\ref{thm:eb-walk} (with a straightforward adjustment to treat the cases of $w, z \leq 0$ ), we recover the right-hand side of~\eqref{eq:HalfSpace_Connexions_UB}.\qed

\subsection{Proof of Lemma~\ref{lem:NonSharpLB}}
\label{sub:NonSharpLB-prf}

We only sketch the proof, as it is a straightforward adaptation of the arguments in~\cite[Section~2.5]{Ott+Velenik-2018(2)}.

Using Lemma~\ref{lem:LB_HalfSpaceToFullSpace} and the (full-space) Ornstein--Zernike asymptotics of~\cite{Campanino+Ioffe+Velenik-2008}, we obtain
\begin{align*}
\label{eq:half-full}
\FKlaw_{\Lambda_+}( v_L \leftrightarrow v_R )
&\geq
\FKlaw\bigl( C_{v_L} \subset \Lambda_+,\, v_L \leftrightarrow v_R \bigr) \\
&=
\FKlaw\bigl( C_{v_L} \subset \Lambda_+ \given v_L \leftrightarrow v_R \bigr) \FKlaw\bigl( v_L \leftrightarrow v_R \bigr) \\
&=
\frac{C}{\sqrt{N}} e^{-2\tau N} \FKlaw\bigl( C_{v_L} \subset \Lambda_+ \given v_L \leftrightarrow v_R \bigr).
\end{align*}
We bound the probability in the right-hand side by restricting to a particular class of paths. Namely, those that connect \(v_L\) to the vertex \(\sfa=(-N+T,N+T)\) by a path going first vertically to \((-N,T)\) and then horizontally to \(\sfa\), and connect \(v_R\) to \(\sfb=(N-T,N+T)\) in a symmetric way (here \(T\) is a fixed large positive number). Arguing as in~\cite[Lemma~2.6]{Ott+Velenik-2018(2)}, we then deduce that
\begin{multline*}
\FKlaw\bigl( C_{v_L} \subset \Lambda_+ \given v_L \leftrightarrow v_R \bigr)\\
\geq
C\,
{\mathbf P}_\sfa (H_{\bbH_-} > H_{\sfb} \given \infty > H_{\sfb})\,
{\mathbf P}_\sfa (\sfS_k > \theta_k\;\forall k\leq H_{\sfb} \given \infty > H_{\bbH_-} > H_{\sfb}).
\end{multline*}
The first probability in the right-hand side can be bounded below by \(C/N\) using~\eqref{eq:tpf} and the local CLT. The reason for the presence of the second probability is that a sufficient condition for the cluster not to visit \(\bbH_-\) is that the diamonds associated to the effective random walk do not intersect \(\bbH_-\). This probability can be shown to be bounded below by a positive constant using the same argument as in~\cite[Lemma~2.7]{Ott+Velenik-2018(2)}.

\subsection{Invariance principle} 
\label{eq:Jscaling}

Recall~\eqref{eq:chib-form-rw}.
Consider the conditional distribution of the excursion $\walk [0, H_{(n , v)}]$ under 
\begin{equation}\label{eq:Cond-P} 
\mathbf{P}_{(0 , u)} (\, \cdot \given H_{(n , v)} <  H_{\bbH_-} ).
\end{equation}
Fix $\epsilon >0$ small. 
In view of Lemma~\ref{lem:EntropicRepulsion}, we need to derive an invariance principle	for Brownian excursion, as $n\to\infty$, \emph{uniformly} in $u,v\in (n^\epsilon, n^{5\epsilon})$.
Namely, let us use $\bfQ^n_{u ,v}$ for the law of the diffusively rescaled linear interpolation $\fre_n$ of the random-walk trajectory $\walk [0, H_{(n, v)}]$; 
\begin{equation}\label{eq:frs} 
	\fre_n = \frI_n \bigl( \walk [0, H_{(n , v)}] \bigr) , 
\end{equation}  
where, given a subset $\{ (t_1,z_1), (t_2,z_2), \dots, (t_k,z_k)\}$ with $t_1 < t_2 < \dots < t_k $, $\frI_n$ is the linear interpolation through the vertices of the rescaled set
\begin{equation}\label{eq:recaled-n} 
\Bigl( \frac1n t_1 , \frac{1}{\sqrt{\chi n}} z_1 \Bigr) , 
\Bigl( \frac1n t_2 , \frac{1}{\sqrt{\chi n}} z_2 \Bigr) , \dots , 
\Bigl( \frac1n t_k , \frac{1}{\sqrt{\chi n}} z_k \Bigr) .
\end{equation}
\begin{theorem} 
\label{thm:conv}
	Let $\bfQ^\infty$ be the law of the positive normalized Brownian excursion $\fre$ on the unit interval $[0,1]$.
	Let $\delta_n\to 0$ arbitrarily slowly as $n\to\infty$ and let $u,v \in ( 0, \delta_n\sqrt{n})\cap \bbN$.  
	Then, the limit as $n\to\infty$ of the family of distributions $\{\bfQ^n_{u,v}\}_{u,v\in (0,\delta_n\sqrt{n})\cap \bbN}$ is equal to $\bfQ^\infty$. More precisely,
	\begin{enumerate} 
		\item The family $\{\bfQ^n_{u,v}\}_{u,v\in (0,\delta_n\sqrt{n})\cap \bbN}$ is tight. 
		\item For any  $k$, any $0 < t_1 < t_2 < \dots < t_k < 1$ and any fixed bounded continuous function $F$ on $\bbR_+^k$, 
		\begin{equation}\label{eq:unif-lim}  
			\lim_{n\to\infty} 
			\bfQ^n_{u_n,v_n} \bigl( F(\fre_n(t_1), \dots , \fre_n(t_k)) \bigr) = 
			\bfQ^\infty \bigl( F(\fre(t_1), \dots , \fre(t_k))\bigr) , 
		\end{equation}
		uniformly in the collections of sequences $\bigl\{ u_n, v_n\in (0, \delta_n\sqrt{n}) \cap \bbN \bigr\}$.
	\end{enumerate}
\end{theorem}

\subsection{Proofs}
\label{sub:proofs}

\begin{proof}[Proof of Theorem~\ref{thm:tpf}]
	First, by the total probability formula,
	\begin{align*}
	\mathbf{P}_{(0 , u)}\left(H_{\mathbb H_-} > H_{(n , v)} \in (0, \infty)\right)
	&=\sum_{k=1}^\infty\mathbf{P}_{(0 , u)}\left(H_{\mathbb H_-} > H_{(n , v)}=k\right)\\
	&=\sum_{k=1}^\infty\mathbf{P}_{(0 , u)}\left(\walk_k=(n,v), H_{\mathbb H_-}>k\right).
	\end{align*}
	Fix some $\varepsilon>0$. Since $\theta$ has finite exponential moments, the exponential Chebyshev inequality implies that
	\[
	\sum_{k<(1/\mathbf{E}\theta-\varepsilon) n}\mathbf{P}_{(0 , u)}\left(\walk_k=(n,v), H_{\mathbb H_-}>k\right)
	\le \sum_{k<(1/\mathbf{E}\theta-\varepsilon) n}\mathbf{P}\Bigl(\sum_{\ell=1}^k\theta_\ell\ge n\Bigr)=\sfO(e^{-c\varepsilon n}).
	\]
	Furthermore, by the same argument for lower tails, we have
	\[
	\sum_{k>(1/\mathbf{E}\theta+\varepsilon) n}\mathbf{P}_{(0 , u)}\left(\walk_k=(n,v), H_{\mathbb H_-}>k\right)
	\le \sum_{k>(1/\mathbf{E}\theta+\varepsilon) n}\mathbf{P}\Bigl(\sum_{\ell=1}^k\theta_\ell\le n\Bigr)=\sfO(e^{-c\varepsilon n}).
	\]
	
	Fix also a large constant $A$. Our next purpose is to estimate the probability $\mathbf{P}_{(0 , u)}\left(\walk_k=(n,v), H_{\mathbb H_-}>k\right)$ for $k\in[(1/\mathbf{E}\theta-\varepsilon) n,n/\mathbf{E}\theta-A\sqrt{n}]$. The main idea is to perform an exponential change of measure:
	\[
	\mathbf{P}^{(h)}(\theta=j,\zeta=x)=\frac{e^{hj}}{\mathbf{E}e^{h\theta}}\mathbf{P}(\theta=j,\zeta=x).
	\]
	Then, clearly,
	\[
	\mathbf{P}_{(0 , u)}\left(\walk_k=(n,v), H_{\mathbb H_-}>k\right)=e^{-h n}\left(\mathbf{E}e^{h\theta}\right)^k
	\mathbf{P}^{(h)}_{(0 , u)}\left(\walk_k=(n,v), H_{\mathbb H_-}>k\right).
	\]
	For all $h$ small enough, we have
	\[
	\mathbf{E}e^{h\theta}\le e^{h\mathbf{E}\theta+h^2\mathbf{Var}(\theta)}.
	\]
	Then, choosing 
	\[
	h_{k,n}=\frac{n-k\mathbf{E}\theta}{2k\mathbf{Var}(\theta)},
	\]
	we arrive at the upper bound
	\begin{equation}
	\label{eq:exp-change}
	\mathbf{P}_{(0 , u)}\left(\walk_k=(n,v), H_{\mathbb H_-}>k\right)\le \exp\left\{-\frac{(n-k\mathbf{E}\theta)^2}{4k\mathbf{Var}(\theta)}\right\}
	\mathbf{P}^{(h_{k,n})}_{(0 , u)}\left(\walk_k=(n,v), H_{\mathbb H_-}>k\right).
	\end{equation}
	
	Define
	\[
	\walk_j^0\defby\walk_j-j(0,\mathbf{E}^{(h_{k,n})}\zeta).
	\]
	Then
	\begin{align*}
	&\Bigl\{\walk_k=(n,v), H_{\mathbb H_-}>k\Bigr\}\\
	&=\Bigl\{u+\sum_{\ell=1}^j\zeta_\ell>0\text{ for all } j\le k, (0,u)+\walk_k=(n,v)\Bigr\}\\
	&=\Bigl\{u+\sum_{\ell=1}^j\zeta_\ell^0>-j\mathbf{E}^{(h_{k,n})}\zeta \text{ for all } j\le k, (0,u)+\walk_k^0=(n,v-n\mathbf{E}^{(h_{k,n})}\zeta)\Bigr\}\\
	&\subseteq\Bigl\{u^0+\sum_{\ell=1}^j\zeta_\ell^0>0 \text{ for all } j\le k, (0,u^0)+\walk_k^0=(n,v^0)\Bigr\},
	\end{align*}
	where
	\[
	u^0\defby u+n\big|\mathbf{E}^{(h_{k,n})}\zeta\big|\text{ and }
	v^0\defby v+n\big|\mathbf{E}^{(h_{k,n})}\zeta\big|-n\mathbf{E}^{(h_{k,n})}\zeta.
	\]
	In other words,
	\[
	\mathbf{P}^{(h_{k,n})}_{(0 , u)}\left(\walk_k=(n,v), H_{\mathbb H_-}>k\right)\le \mathbf{P}^{(h_{k,n})}_{(0 , u^0)}\left(\walk_k^0=(n,{v^0}), H^0_{\mathbb H_-}>k\right), 
	\]
	where $H^0_{\mathbb H_-}$ is the first hitting time of 
		$\mathbb H_-$ by the modified random walk $\walk^0 = 
		\lb \sfT^0 , \sfZ^0\rb$. 
	Since 
	$H^0_{\mathbb H_-}$ is an exit time for a one-dimensional random walk 
	$\bfZ^0$ 
	with zero mean and finite variance, one has the bound (see~\cite[Lemma 2.1]{AGKV})
	\[
	\mathbf{P}^{(h_{k,n})}_{(0 , z)}\left(H^0_{\mathbb H_-}>k\right)\le c_1\frac{z+1}{\sqrt{k}}
	\]
	uniformly in all $z>0$. 
	
	Using this bound in the proof of~\cite[Lemma 28]{DW15}, one gets easily the bound
	\[
	\mathbf{P}^{(h_{k,n})}_{(0 , u^0)}\left(\walk_k^0=(n,{v^0}), H^0_{\mathbb H_-}>k\right)
	\le c_2\frac{(u^0+1)(v^0+1)}{k^2}
	\]
	uniformly in all positive $u^0,v^0$.
	
	Recall that $\theta$ and $\zeta$ are uncorrelated. 
	Then, by the Taylor formula,
	\[
	\mathbf{E}^{(h)}\zeta
	=\frac{\mathbf{E}(\zeta e^{h\theta})}{\mathbf{E}e^{h\theta}}
	=\frac{h^2}{2}\mathbf{E}(\theta^2\zeta)+\sfo(h^2),\quad h\to0. 
	\]
	Therefore, for small \(h_{k,n}\),
	\[
	\abs{\mathbf{E}^{(h_{k,n})}\zeta} \le a h^2_{k,n}
	\quad\text{with}\  a=\mathbf{E}|\theta^2\zeta|.
	\]
	As a result, we have
	\[
	\mathbf{P}^{(h_{k,n})}_{(0 , u)}\left(\walk_k=(n,v), H_{\mathbb H_-}>k\right)\le c_3\frac{(u+nh^2_{k,n})(v+nh^2_{k,n})}{k^2}.
	\]
	Combining this bound with \eqref{eq:exp-change}, summing over $k$ and using the fact that the functions $h^\pm$ are asymptotically linear, we obtain
	\begin{equation}
	\label{eq:moderate1}
	\sum_{k\in[(1/\mathbf{E}\theta-\varepsilon)n,n/\mathbf{E}\theta-A\sqrt{n}]}\mathbf{P}_{(0 , u)}\left(\walk_k=(n,v), H_{\mathbb H_-}>k\right)\le \frac{f_1(A)h^+(u)h^-(v)}{n^{3/2}},
	\end{equation}
	where $f_1(A)\to0$ as $A\to\infty $. 
	This estimate is uniform in $u, v \in (0, \sqrt{n})\cap \bbN$.,
	
	The same argument gives, 
	 also
	 uniformly in $u, v \in (0, \sqrt{n})\cap \bbN$, 
	\begin{equation}
	\label{eq:moderate2}
	\sum_{k\in[n/\mathbf{E}\theta+A\sqrt{n},(1/\mathbf{E}\theta+\varepsilon)n]}\mathbf{P}_{(0 , u)}\left(\walk_k=(n,v), H_{\mathbb H_-}>k\right)\le \frac{f_2(A)h^+(u)h^-(v)}{n^{3/2}},
	\end{equation}
	where $f_2(A)\to0$ as $A\to\infty $.
	
	For $k\in[n/\mathbf{E}\theta-A\sqrt{n},n/\mathbf{E}\theta+A\sqrt{n}]$
	one can repeat the proof of the local limit theorems from~\cite{DW15}.
	Compared to that paper, we have a rather particular case: a two-dimensional random walk confined to the upper half-plane. But we want to get a result which is valid not only for bounded start- and endpoints. Since we have a walk in the upper half-plane, the corresponding harmonic function depends on the second coordinate only and is equal to the harmonic function of the walk $Z_n$ killed at leaving $(0,\infty)$. So, we only have to show that the convergence in \cite[Lemma~21]{DW15} holds for all starting points $(0,u)$  with $u\le\delta_n\sqrt{n}$. More precisely, we need to prove that
	\begin{equation}
	\label{eq:lemma21}
	\mathbf{E}_{(0,u)}[Z_{\nu_k}; H_{\mathbb H_-}>\nu_k,\nu_k\le k^{1-\varepsilon}]=h^+(u)(1+\sfo(1)) 
	\end{equation}
	uniformly in $u\leq \delta_n\sqrt{n}$ and $k\in[n/\mathbf{E}\theta-A\sqrt{n},n/\mathbf{E}\theta+A\sqrt{n}]$.
	Above, $\nu_k$ is the first hitting time of the positive half-space $(k^{1/2-\epsilon}, 0)+\bbH_+$.
	The relation~\eqref{eq:lemma21} leads to the fact that all the arguments in~\cite[Sections~4 and~5]{DW15} hold uniformly in $u\in(0,\delta_n\sqrt{n})$.  
	Then, repeating the proof in~\cite[Theorem 6]{DW15}, we obtain
	\[
	\mathbf{P}_{(0 , u)}\left(\walk_k=(n,v), H_{\mathbb H_-}>k\right)
	\sim c_4\frac{h^+(u)h^-(v)}{k^2}\exp\left\{-\frac{(n-k\mathbf{E}\theta)^2}{2k\mathbf{Var}(\theta)}\right\},
	\]
	uniformly in
	$u, v \in (0, \delta_n\sqrt{n})\cap \bbN$. 
	Summing over $k$, we get
	\begin{equation}\label{eq:k-prop}
	\sum_{k\in[n/\mathbf{E}\theta-A\sqrt{n},n/\mathbf{E}\theta+A\sqrt{n}]}\mathbf{P}_{(0 , u)}\left(\walk_k=(n,v), H_{\mathbb H_-}>k\right)
	\sim(C_5-f_3(A))\frac{h^+(u)h^-(v)}{n^{3/2}},
	\end{equation}
	where $f_3(A)\to0$ as $A\to\infty $.
	
	Combining all the estimates above, we finally deduce the asymptotic relation~\eqref{eq:tpf}. Thus, it remains to prove~\eqref{eq:lemma21}. Here one can use again the fact that we are dealing with a one-dimensional random walk. Since $Z_n$ is a martingale, we use the optional stopping theorem to obtain
	\[
	u=\mathbf{E}_{(0,u)}Z_{\nu_k\wedge H_{\mathbb H_-}}
	=\mathbf{E}_{(0,u)}[Z_{\nu_k};\nu_k< H_{\mathbb H_-}]
	+\mathbf{E}_{(0,u)}[Z_{H_{\mathbb H_-}};\nu_k\ge H_{\mathbb H_-}].
	\]
	Consequently,
	\begin{multline}
	\label{eq:lemma21.1}
	\mathbf{E}_{(0,u)}[Z_{\nu_k};\nu_k< H_{\mathbb H_-},\nu_k\le k^{1-\epsilon}]\\
	= u-\mathbf{E}_{(0,u)}[Z_{H_{\mathbb H_-}};\nu_k\ge H_{\mathbb H_-}]-\mathbf{E}_{(0,u)}[Z_{\nu_k};\nu_k< H_{\mathbb H_-},\nu_k> k^{1-\epsilon}].
	\end{multline}
	Recalling that $\mathbf{E}_{(0,u)}H_{\mathbb H_-}$ is bounded and that $h^+(u)\sim u$ as $u\to\infty$, one gets easily
	\begin{equation}
	\label{eq:lemma21.2}
	u-\mathbf{E}_{(0,u)}[Z_{H_{\mathbb H_-}};\nu_k\ge H_{\mathbb H_-}]=h^+(u)(1+\sfo(1))
	\end{equation}
	uniformly in $u$. 
	Furthermore, by the Cauchy--Schwarz inequality,
	\[
	\mathbf{E}_{(0,u)}[Z_{\nu_k};\nu_k< H_{\mathbb H_-},\nu_k> k^{1-\epsilon}]
	\le \mathbf{E}^{1/2}_{(0,u)}[Z^2_{\nu_k};\nu_k< H_{\mathbb H_-}]
	\mathbf{P}^{1/2}_{(0,u)}(\nu_k> k^{1-\epsilon}, \nu_k< H_{\mathbb H_-}). 
	\]
	Obviously,
	$Z^2_{\nu_k}=(Z_{\nu_k-1}+\zeta_{\nu_k})^2\le 2 k^{1-2\varepsilon}+2\zeta_{\nu_k}^2$ on the event 
	$\{\nu_k< H_{\mathbb H_-}\}$. Thus, using the total probability formula, we get
	\[
	\mathbf{E}_{(0,u)}[Z^2_{\nu_k};\nu_k< H_{\mathbb H_-}]
	\le 2(k^{1-2\epsilon}+\mathbf{E}\zeta)
	\mathbf{E}_{(0,u)}[\nu_k\wedge H_{\mathbb H_-}]
	\le C k^{2-4\epsilon}.
	\]
	In the last step, we have used the bound $\mathbf{E}_{(0,u)}[\nu_k\wedge H_{\mathbb H_-}]\le C k^{1-2\epsilon}$, which follows from the normal approximation. By~\cite[Lemma~14]{DW15},
	\begin{align*}
	\mathbf{P}_{(0,u)}(\nu_k> k^{1-\epsilon}, \nu_k< H_{\mathbb H_-})
	&\le \mathbf{P}_{(0,u)}(\nu_k> k^{1-\epsilon}, H_{\mathbb H_-}> k^{1-\epsilon})\le e^{-Ck^\epsilon}.
	\end{align*}
	As a result,
	\begin{equation}
	\label{eq:lemma21.3}
	\mathbf{E}_{(0,u)}[Z_{\nu_k};\nu_k< H_{\mathbb H_-},\nu_k> k^{1-\epsilon}]=\sfO(e^{-Ck^\epsilon})
	\end{equation}
    for some $C>0$. Combining~\eqref{eq:lemma21.1}--\eqref{eq:lemma21.3}, we obtain~\eqref{eq:lemma21}. 
	
	The derivations of~\eqref{eq:tpf-a} and~\eqref{eq:tpf-b} are very similar and even simpler and are thus omitted.	
	\end{proof}
		
	\begin{proof}[Proof of Theorem~\ref{thm:eb-walk}]
	Let us introduce some provisional notation: 

	\smallskip 
	\noindent 
	\textsf{Hitting times.}  
	$\bbH^k_- \defby 
	(0, -k )+ \bbH_- = \lbr x = (x_1 , x_2)~:~ x_2 <- k\rbr$ for the negative half-planes passing through the shifted points $(0 , -k )$. 
	
	\smallskip 
	\noindent 		
	\textsf{Minimal heights.} 
	Given an in general random time $H\in \bbN$, let 
	$\sfZ_* ( H)\defby  \min_{\ell  = 0, \dots , H} \sfZ_\ell$ be the minimal value of the vertical coordinate $\sfZ$ of the random-walk trajectory $\sfS [0,H]$ on the time interval $[0, H]$. Furthermore, 
	let $m_* (H ) = \min\setof{m}{ (m,\sfZ_*(H)) \in \sfS[0,H]}$ be the horizontal projection of the leftmost vertex of $\sfS [0, H]$, at which the minimal height  $\sfZ_* ( H)$ was attained. 
		
	Evidently, 
	\begin{multline} 
	\label{eq:eb-walk-sum} 
	{\mathbf E}_{(0,u)}\bigl( \IF{H_{(n,v)} <\infty} \delta^{\calL_{\bbH_-}( H_{(n , v)})}\bigr)
	\leq\\
	{\mathbf P}_{(0,u)} (H_{(n , v)} < H_{\bbH_-})
	 + 
	\sum_{k=0}^{\infty} {\mathbf E}_{(0,u)}\bigl( \IF{H_{(n,v)}<\infty} 
    \IF{Z^* (H_{(n,v)}) =-k} \delta^{\calL_{\bbH_-} (H_{(n , v)})}\bigr).
	\end{multline}
	The first term on the right-hand side above is controlled by Theorem~\ref{thm:tpf}. 
	In view of the exponential tails, we may fix $\epsilon >0$ small and  restrict attention to 
	such terms in the above sum, which satisfy $k\leq n^{1/2+ \epsilon}$. 
	
	Now, 
	\begin{align}
	\label{eq:E-splitn} 
	{\mathbf E}_{(0,u)}\bigl( \IF{H_{(n,v)}<\infty}& 
		\IF{Z^* (H_{(n,v)}) =-k} \delta^{\calL_{\bbH_-}(H_{(n , v)})}\bigr) \\
		\nonumber
		&=
	{\mathbf E}_{(0,u)}\bigl( \IF{m_*\in [0,n/2 ]} \IF{H_{(n,v)}<\infty} 
		\IF{Z^* (H_{(n,v)}) =-k} \delta^{\calL_{\bbH_-} (H_{(n , v)})}\bigr)\\
		\nonumber 
		&
           + 
    {\mathbf E}_{(0,u)}\bigl(\IF{m_*\in [n/2 ,n ]} \IF{H_{(n,v)}<\infty} 
		\IF{Z^* (H_{(n,v)}) =-k} \delta^{\calL_{\bbH_-} (H_{(n , v)})}\bigr).
	\end{align}
	We shall consider only the first term on the right-hand side above, the second one is completely similar. 
	Let us decompose  with respect to the possible values of $m_*$ 
	\begin{multline*}
	{\mathbf E}_{(0,u)}\bigl(\IF{m_*\in [0,n/2 ]} \IF{H_{(n,v)}<\infty} 
		\IF{Z^* (H_{(n,v)}) =-k} \delta^{\calL_{\bbH_-}(H_{(n , v)})}\bigr)  \\
		= \sum_{m=1}^{\lfloor n/2\rfloor}
	{\mathbf E}_{(0,u)}\bigl(\IF{m_* = m} \IF{H_{(n,v)}<\infty} 
		\IF{Z^* (H_{(n,v)}) =-k} \delta^{\calL_{\bbH_-} (H_{(n , v)})}\bigr) .
	\end{multline*} 
	We shall rely on several crude upper bounds. The first one is  
	\begin{align} 
	\label{eq:crude-bound1} 
	&{\mathbf E}_{(0,u)}\bigl(\IF{m_* = m} \IF{H_{(n,v)}<\infty} 
		\IF{Z^* (H_{(n,v)}) =-k} \delta^{\calL_{\bbH_-} (H_{(n , v)})}\bigr)  
		\nonumber\\
	&\quad\leq 
		{\mathbf E}
		_{(0,u)}
		\Bigl(  
		\IF{H_{(m , -k) } = H_{\bbH^{k-1}_-}} 
		{\mathbf E}_{(m, -k)}
		\bigl( 
		\IF{H_{(n,v)} <H_{\bbH^k_-}}  \delta^{\calL_{\bbH_-} (H_{(n , v)})}
		\bigr) 
		\Bigr)
		\nonumber\\
	&\quad\leq 
		{ {\rm e}^{-cn  \wedge \frac{n^2}{k}}} + 
		{\mathbf E}
		_{(0,u)}
		\Bigl(
		\IF{H_{(m , -k) } = H_{\bbH^{k-1}_-}} 
		{\mathbf E}_{(m, -k)}
		\bigl(
		\delta^{\calL_{\bbH_-}(k)} 
		{\mathbf P}_{{\mathsf S}_k } 
		(H_{(n,v)}< H_{\bbH^k_-})
		\bigr) 
		\Bigr)  .
	\end{align}
	For $k\leq n^{1/2+ \epsilon}$, the first summand in~\eqref{eq:crude-bound1} above is negligible. 
	We claim that there exist $c, C \in (0, \infty )$ such that \footnote{The stretched $\sqrt{k}$ rate of decay is used only for minimizing the discussion needed for ruling out $k > \sqrt{n}$. For the rest of $k$-s, the usual exponential bounds with decay rate proportional to $k$ hold.} 
	\begin{equation}\label{eq:k-c-bound} 
	{\mathbf E}_{(m, -k)}
		\bigl(
		\delta^{\calL_{\bbH_-}(k)} 
		{\mathbf P}_{{\mathsf S}_k } 
		(H_{(n,v)}< H_{\bbH^k_-})  
		\bigr)
		\leq C {\rm e}^{-c\sqrt{k} } 
		{\mathbf P}_{(m, -k)} (H_{(n,v)}< H_{\bbH^k_-}), 
	\end{equation} 
	uniformly in $n\in \bbN$ sufficiently large and, then, in $v\leq \sqrt{n}$, $m\in [0, n/2]$ and (for $\epsilon >0$ being fixed appropriately  small) $k \in [0, n^{1/2 +\epsilon} ]$. 
	We shall relegate the justification of~\eqref{eq:k-c-bound} to the end of the proof. At this stage, note that~\eqref{eq:k-c-bound} (and its analogue for the second term on the right-hand side of~\eqref{eq:E-splitn}) would imply that 
	\begin{multline} 
	\label{eq:m-k-bound} 
	{\mathbf E}_{(0,u)}\bigl(
	\IF{H_{(n,v)}<\infty} 
		\IF{Z^* (H_{(n,v)}) =-k} \delta^{\calL_{\bbH_-} (H_{(n , v)})}\bigr)    \\
		\leq C {\rm e}^{-c\sqrt{k} } 
	{\mathbf P}_{(0,u)}\bigl( 
	H_{(n,v)}<\infty  ; 
	Z^* (H_{(n,v)}) =-k  
		\bigr).
	\end{multline}
	It follows that, as far as the sum in~\eqref{eq:eb-walk-sum} is concerned, we may further restrict attention to $k\leq \frac{1}{c}(\log n )^3$. In the latter case, however, Theorem~\ref{thm:tpf} applies and 
	\begin{equation}\label{eq:pf-k-bound} 
	{\mathbf P}_{(0,u)} (H_{(n,v)} < H_{\bbH^{k}_-})
	\sim  C\frac{(u+k )(v+k )}{n^{3/2}} .
	\end{equation}
	Consequently, 
	\begin{multline} 
	\label{eq:pf-k-bound-leq} 
	{\mathbf P}_{(0,u)}
	\bigl( 
	H_{(n,v)}<\infty ; 
	\sfZ_* (H_{(n,v)}) =-k 
	\bigr)\\
	= {\mathbf P}_{(0,u)} (H_{(n,v)} < H_{\bbH^{k+1}_-})
	-
	{\mathbf P}_{(0,u)}	(H_{(n,v)} < H_{\bbH^{k}_-})\\
	\leq C\frac{(u+k+1 )(v+k+1 )}{n^{3/2}} .
	\end{multline}	
	Substituting~\eqref{eq:m-k-bound} and~\eqref{eq:pf-k-bound-leq} into~\eqref{eq:eb-walk-sum}
	 yields: There exist $c , C \in (0,\infty)$, such that 
	\begin{equation}\label{eq:P-k-crude-fin}
	{\mathbf E}_{(0,u)}
	\bigl(
	\IF{H_{(n,v)}<\infty} 
	\IF{Z^* (H_{(n,v)}) =-k} 
	\delta^{\calL_{\bbH_-}(\sfH_{(n , v)})}
	\bigr) 
	\leq  
	\sum_{k=0}^\infty C\frac{(u+k )(v+k ) {\rm e}^{-c\sqrt{k}}}{n^{3/2}} , 
	\end{equation}
	and we are home.  
	\smallskip 
	
	\noindent 
	\textit{Proof of~\eqref{eq:k-c-bound}}. First of all, in view of Theorem~\ref{thm:tpf}, the right-hand side of~\eqref{eq:k-c-bound} satisfies 
	\begin{equation}\label{eq:RHS} 
	{\mathbf P}_{(m, -k)} (H_{(n,v)}< H_{\bbH^k_-}) \geq 
			C\frac{(v+\min\lbr k , \sqrt{n}\rbr )}{n^{3/2}} , 
	\end{equation} 
	uniformly in $m$ and $k$ in question. Consider now the left-hand side of~\eqref{eq:k-c-bound}.
	Since $k\leq n^{1/2 +\epsilon}$ and $\epsilon$ is small, we  may rely on moderate deviation estimates and 
	restrict attention to $\abs{S_k  - (m, -k )} = \babs{\sum_1^k \zeta_i}\leq \sqrt{n}$.
	In the latter case Theorem~\ref{thm:tpf} applies, and the following upper bound holds: There exists $C^*<\infty$, such that 
	\begin{equation}\label{eq:k-terms} 
		{\mathbf E}_{(m, -k)}
			\bigl(  
			\delta^{\calL_{\bbH_-}(k)} 
			{\mathbf P}_{{\mathsf S}_k} ({H_{(n,v)}< H_{\bbH^k_-}})
			\bigr)  
			\leq 
			C^* {\mathbf E}_{(m, -k)}
			\Bigl(  
			\delta^{\calL_{\bbH_-}(k)} 
            \frac{ v + \babs{\sum_1^k \zeta_i}}{n^{3/2}}\Bigr)  .
	\end{equation}
	It remains to notice that, by the usual large deviation upper bounds under Cramér's condition, there exists $c^*>0$ such that 
	\begin{equation}\label{eq:c-star-bound} 
	{\mathbf E}_{(m, -k)}
			\Bigl(
			\delta^{\calL_{\bbH_-}(k)} 
            \bigl( v + \babs{\sum_1^k \zeta_i}\bigr)\Bigr) 
            \leq C^* {\rm e}^{-c^* k}  (v + k ) , 
	\end{equation} 
	uniformly in $k, v\in \bbZ_+$. Together with~\eqref{eq:RHS}, this implies~\eqref{eq:k-c-bound}.
\end{proof}

\begin{proof}[Proof of Theorem~\ref{thm:conv}]
	The above changes in the arguments from~\cite{DW15} allow one to repeat the proof of~\cite[Theorem 6]{DW}, which gives the convergence of a properly centered and rescaled walk $\sf S_n$ towards the two-dimensional Brownian bridge conditioned to stay in the upper half-plane. This convergence is uniform in the range of $u,v$ as formulated  in Theorem~\ref{thm:conv}. In particular, we have convergence of each coordinate of the two-dimensional walk $\sf S_n$. More precisely, again uniformly in $u,v \in (0,\delta_n\sqrt{n})\cap \bbN$ and, also for each $A$ fixed, uniformly in the number of steps  $k\in[n/\mathbf{E}\theta-A\sqrt{n},n/\mathbf{E}\theta+A\sqrt{n}]\cap \bbN$ which shows up in the principal sum~\eqref{eq:k-prop}, 
	\begin{equation}
	\label{eq:inv1}
		\mathbf{P}_{(0,u)}\bigl(\max_{j\le k}|\sfT_j-j\mathbf{E}\theta|>\delta k\bigm|\walk_k=(n,v), H_{\mathbb H_-}>k\bigr)\to 0,
	\end{equation}
	and, for any $\ell\in \bbN$, any $0 < t_1 < t_2 <\dots < t_\ell <1$, 
	any fixed bounded continuous function $F$  on $\bbR_+^\ell$, 
	\begin{equation}
	\label{eq:inv2}
	\mathbf{E}_{(0,u)}\left[F\lb \frz_{k} (t_1 ) , \dots , \frz_k (t_\ell )\rb \big|\walk_k=(n,v), H_{\mathbb H_-}>k\right]\to \mathbf{Q}^\infty \left[F\lb \frz_{k} (t_1 ) , \dots , \frz_k (t_\ell )\rb\right] ,
	\end{equation}
	where $\frz_{k}$ is the linear interpolation with nodes 
	\[
	\Bigl( \frac1k,\frac{{\sf Z}_1}{\sqrt{k\mathbf{Var}(\theta))}}\Bigr), 
	\Bigl( \frac2k,\frac{{\sf Z}_2}{\sqrt{k\mathbf{Var}(\theta))}}\Bigr) , 
	\dots, 
	\Bigl( \frac{k-1}{k},\frac{{\sf Z}_{k-1}}{\sqrt{k\mathbf{Var}(\theta))}}\Bigr) , 
	\Bigl( 1, \frac{v}{\sqrt{k\mathbf{Var}(\theta))}}\Bigr) .
	\]
	Thus, in view of~\eqref{eq:moderate1} and~\eqref{eq:moderate2}, it remains to bound the difference between this interpolation and the interpolation in~\eqref{eq:recaled-n} for $k$ such that $\abs{n-k\mathbf{E}\theta}\le A\sqrt{n}$. To this end, we notice that the random change of time $h_k$, defined as the linear interpolation of $(\ell/k,T_\ell/n)$, transforms~\eqref{eq:recaled-n} into $\frz_{k}$. Combining this observation with~\eqref{eq:inv1} and~\eqref{eq:inv2}, we obtain the convergence of~\eqref{eq:recaled-n} in the Skorokhod $J_1$-topology. Since the limiting process --- Brownian excursion --- has continuous paths, one has also  the convergence in the uniform topology. This follows from Theorem~2.6.2 in Skorokhod's classical paper~\cite{Skor56}.
\end{proof}

\section*{Acknowledgments}

The research of D.~Ioffe was partially supported by 
Israeli Science Foundation grant 765/18,
S.~Ott was supported by the Swiss NSF through an early Postdoc.Mobility Grant and
Y.~Velenik acknowledges support of the Swiss NSF through the NCCR
SwissMAP.

\appendix

\section{A Monotone Coupling}
\label{app:coupling}

For \(\Delta\subset E_{\Ztwo}\) finite, denote \(\FKlaw_{a,\Delta}\equiv\FKlaw_{a,\Delta}^0\) the random-cluster measure in \(\Delta\) with free (\(0\)) boundary condition and weights \(e^{\beta}-1\) on edges with both endpoints having nonnegative second coordinate and weight \(a\) on the others. In particular, \(\FKlaw_{0,\Lambda}\) is the random-cluster measure on the half-box \(\Lambda_+\) with free boundary condition and weights \(e^{\beta}-1\).

In this section, we construct a monotone coupling of \(\FKlaw_{b,\Delta}\) and \(\FKlaw_{a,\Delta}\) for \(b>a\). The construction follows closely the one used in the proof of~\cite[Theorem 3.47]{Grimmett-2006}. We fix \(\Delta\) and let \(\Delta^+ = \Delta\cap (\R\times\R_{\geq 0}) \) and \(\Delta^- = \Delta\cap (\R\times\R_{< 0})\); both are seen as the graphs induced by their set of edges, where edges are identified with the corresponding \emph{open} line segments. For a finite set of edges \(E\), denote by \(\ouvert_{-}(E)\) the number of edges in \(E\) with at least one endpoint having negative second coordinate.

Let \(e_1,\dots,e_{\abs{E_{\Delta}}}\) be an enumeration of the edges of \(\Delta\) and set \(E_i=\{e_1,\dots,e_i\}\). Let \((U_i)_{i=1}^{\abs{E_{\Delta}}}\) be an \iid family of uniform random variables on \([0,1]\). From a realization \(u=(u_i)_{i}\) of \(U=(U_i)_i\), we construct two configurations \(\omega=\omega(u)\) and \(\eta=\eta(u)\)
with joint distribution  \(\coupling\) as follows:

\smallskip
\begin{algorithm}[H]
	\DontPrintSemicolon
	\label{alg:Sampling}
	Set \(i=1\) \;
	\While{\(i\leq |E_{\Delta}|\)}{
		Set \(\omega_{e_i} =\IF{u_i < \FKlaw_{b}( X_{e_i}=1\given X_{E_{i-1}} = \omega_{E_{i-1}} )} \)\;
		Set \(\eta_{e_i} =\IF{u_i< \FKlaw_{a}( X_{e_i}=1 \given X_{E_{i-1}} = \eta_{E_{i-1}} )} \)\;
		Update \(i=i+1\)\;
	}
	\caption{Constructing \(\omega,\eta\).}
\end{algorithm}

\smallskip
Monotonicity of random-cluster measures in their parameters and boundary condition ensures that \(\omega\geq \eta\). Direct computation shows that \(\omega(U)\sim \FKlaw_{b}\) and \(\eta(U)\sim \FKlaw_{a}\).
\begin{claim}
	For any \(e_M\in\Delta^{-}\),
	\begin{equation*}
	\coupling(\omega_{e_M}=1,\, \eta_{e_M}=0 \given U_1=u_1,\dots,U_{M-1}=u_{M-1} ) \geq \frac{b-a}{(b+q)(b+1)}
	\end{equation*}
    uniformly over \(u_1,\dots,u_{M-1}\).
\end{claim}
\begin{proof}
	First, notice that (denoting \(\omega_{E_{M-1}}(u_1,\dots,u_{M-1})\) the configuration \(\omega\) restricted to \(E_{M-1}\) and similarly for \(\eta\))
	\begin{align*}
	\coupling(\omega_{e_M}=1,\, \eta_{e_M}=0& \given U_1=u_1,\dots,U_{M-1}=u_{M-1} )\\
	&=
	\FKlaw_{b}( X_{e_M}=1 \given X_{E_{i-1}} = \omega_{E_{i-1}} ) - \FKlaw_{a}( X_{e_M}=1 \given X_{E_{i-1}} = \eta_{E_{i-1}} )\\
	&\geq
	\FKlaw_{b}( X_{e_M}=1 \given X_{E_{i-1}} = \omega_{E_{i-1}} ) - \FKlaw_{a}( X_{e_M}=1 \given X_{E_{i-1}} = \omega_{E_{i-1}} )\\
	&=\int_{a}^{b} \frac{\dd}{\dd s} \FKlaw_{s}( X_{e_M}=1 \given X_{E_{i-1}} = \omega_{E_{i-1}} )\, \dd s.
	\end{align*}
	The claim will thus follow once we establish that \(\frac{\dd}{\dd s} \FKlaw_{s}( X_{e_M}=1\given X_{E_{i-1}} = \omega_{E_{i-1}} )\geq (b+q)^{-1}(b+1)^{-1}\) for any \(s\leq b\). Write \(\FKlaw_{s}^{*}(\cdot) = \FKlaw_{s}( \cdot \given X_{E_{i-1}} = \omega_{E_{i-1}} ) \); this is a random-cluster measure on \( E_{\Lambda}\setminus E_{M-1}\). Let \(X\sim \FKlaw_{s}^{*}\). Then,
	\begin{align*}
	\frac{\dd}{\dd s} \FKlaw_{s}^{*}( X_{e_M}=1)
    &=
    \frac{1}{s} \Cov_s^{*} \bigl( \abs{\ouvert_{-}(X)},\, X_{e_M} \bigr)\\
	&=
    \frac{1}{s} \Cov_s^{*} \bigl( \abs{\ouvert_{-}(X)} - X_{e_M},\, X_{e_M} \bigr) +
    \frac{1}{s} \FKlaw_{s}^{*}(X_{e_M}=1) \FKlaw_{s}^{*}(X_{e_M}=0)\\
	&\geq
    \frac{1}{s} \frac{s}{s+q} \frac{1}{s+1}
    \geq
    \frac{1}{(b+q)(b+1)},
	\end{align*}
	since \(\abs{\ouvert_{-}(X)} - X_{e_M}\) is a nondecreasing function and is thus positively correlated with \(X_{e_M}\) (the remainder follows from finite energy).
\end{proof}

As \(\coupling(\omega_{e_M}=1 \given U_1=u_1,\dots,U_{M-1}=u_{M-1} )\leq \frac{b}{1+b}\) (by finite energy),
 one has 
\[
\coupling(\eta_{e_M}=0 \given \omega_{e_M}=1,\, U_1=u_1,\dots,U_{M-1}=u_{M-1} ) \geq \frac{b-a}{(b+q)(b+1)} \frac{1+b}{b} = \frac{b-a}{(b+q)b} .
\]

Write \(\epsilon=\epsilon(a,b) = \frac{b-a}{(b+q)b}\). This implies that, for any configuration \(\psi\) and any set \(A\subset E_{\Delta^-}\) with \(\psi_e=1\) for all \(e\in A\),
\begin{equation}
\label{eq:strictMonotCoupling}
\coupling( \omega = \psi,\, \eta_e=1 \forall e\in A) \leq (1-\epsilon)^{|A|}\, \FKlaw_{b}(\psi).
\end{equation}
Indeed, writing \(D_i=\{\eta_{e_i}=1\}\) if \(e_i\in A\) and \(D_i=\{\eta_{e_i}\in\{0,1\}\}\) otherwise and setting \(D_{E_i}=\bigcap_{j\leq i} D_j\), we get
\begin{align*}
\frac{\coupling( \omega = \psi,\, \eta_e=1 \forall e\in A)}{\coupling(\omega = \psi)} &\leq \prod_{i=1}^{\abs{E_{\Lambda}}}\frac{\coupling( \omega_{e_i} = \psi_{e_i},\, D_i \given \omega_{E_{i-1}} =\psi_{E_{i-1}},\, D_{E_{i-1}})}{\coupling(\omega_{e_i} = \psi_{e_i} \given \omega_{E_{i-1}} =\psi_{E_{i-1}})}\\
&\leq
\prod_{i:\, e_i\in A} \coupling( \eta_{e_i}=1 \given \omega_{e_i} = 1,\, \omega_{E_{i-1}}=\psi_{E_{i-1}},\, D_{E_{i-1}})\\
&\leq
(1-\epsilon)^{|A|}.
\end{align*}

\bibliographystyle{plain}
\bibliography{PottsInterfaceWithWallFull}

\begin{thebibliography}{10}

\bibitem{Abraham+Reed-1974}
D.~B. Abraham and P.~Reed.
\newblock Phase separation in the two-dimensional {I}sing ferromagnet.
\newblock {\em Phys. Rev. Lett.}, 33:377--379, Aug 1974.

\bibitem{AGKV}
V.~I. Afanasyev, J.~Geiger, G.~Kersting, and V.~A. Vatutin.
\newblock Criticality for branching processes in random environment.
\newblock {\em Ann. Probab.}, 33(2):645--673, 2005.

\bibitem{Beffara+Duminil-Copin-2012}
V.~Beffara and H.~Duminil-Copin.
\newblock The self-dual point of the two-dimensional random-cluster model is
  critical for {$q\geq 1$}.
\newblock {\em Probab. Theory Related Fields}, 153(3-4):511--542, 2012.

\bibitem{Bricmont+Lebowitz+Pfister-1981}
J.~Bricmont, J.~L. Lebowitz, and C.~E. Pfister.
\newblock On the local structure of the phase separation line in the
  two-dimensional {I}sing system.
\newblock {\em J. Statist. Phys.}, 26(2):313--332, 1981.

\bibitem{Campanino+Ioffe-2002}
M.~Campanino and D.~Ioffe.
\newblock Ornstein-{Z}ernike theory for the {B}ernoulli bond percolation on
  {$\Bbb Z^d$}.
\newblock {\em Ann. Probab.}, 30(2):652--682, 2002.

\bibitem{Campanino+Ioffe+Louidor-2010}
M.~Campanino, D.~Ioffe, and O.~Louidor.
\newblock Finite connections for supercritical {B}ernoulli bond percolation in
  2{D}.
\newblock {\em Markov Process. Related Fields}, 16(2):225--266, 2010.

\bibitem{Campanino+Ioffe+Velenik-2003}
M.~Campanino, D.~Ioffe, and Y.~Velenik.
\newblock Ornstein-{Z}ernike theory for finite range {I}sing models above
  {$T_c$}.
\newblock {\em Probab. Theory Related Fields}, 125(3):305--349, 2003.

\bibitem{Campanino+Ioffe+Velenik-2008}
M.~Campanino, D.~Ioffe, and Y.~Velenik.
\newblock Fluctuation theory of connectivities for subcritical random cluster
  models.
\newblock {\em Ann. Probab.}, 36(4):1287--1321, 2008.

\bibitem{DIMP}
A.~De~Masi, D.~Ioffe, I.~Merola, and E.~Presutti.
\newblock Metastability and uphill diffusion.
\newblock Provisional title, in preparation.

\bibitem{DW15}
D.~Denisov and V.~Wachtel.
\newblock Random walks in cones.
\newblock {\em The Annals of Probability}, 43(3):992--1044, 2015.

\bibitem{Dobrushin-1992}
R.~Dobrushin.
\newblock A statistical behaviour of shapes of boundaries of phases.
\newblock In R.~Kotecký, editor, {\em Phase Transitions: Mathematics, Physics,
  Biology…}, pages 60--70. 1992.

\bibitem{Dobrushin+Kotecky+Shlosman-1992}
R.~Dobrushin, R.~Koteck\'{y}, and S.~Shlosman.
\newblock {\em Wulff construction}, volume 104 of {\em Translations of
  Mathematical Monographs}.
\newblock American Mathematical Society, Providence, RI, 1992.

\bibitem{Doney98}
R.~A. Doney.
\newblock The martin boundary and ratio limit theorems for killed random walks.
\newblock {\em Journal of the London Mathematical Society}, 58(3):761--768,
  1998.

\bibitem{Duminil-Copin+Manolescu-2016}
H.~Duminil-Copin and I.~Manolescu.
\newblock The phase transitions of the planar random-cluster and potts models
  with $q\geq 1$ are sharp.
\newblock {\em Probability Theory and Related Fields}, 164(3):865--892, 2016.

\bibitem{DW}
J.~Duraj and V.~Wachtel.
\newblock Invariance principles for random walks in cones.
\newblock {\em arXiv:1508.07966}, 2015.

\bibitem{Durrett-1979}
R.~Durrett.
\newblock On the shape of a random string.
\newblock {\em Ann. Probab.}, 7(6):1014--1027, 1979.

\bibitem{Gallavotti-1972}
G.~Gallavotti.
\newblock The phase separation line in the two-dimensional {I}sing model.
\newblock {\em Comm. Math. Phys.}, 27:103--136, 1972.

\bibitem{Greenberg+Ioffe-2005}
L.~Greenberg and D.~Ioffe.
\newblock On an invariance principle for phase separation lines.
\newblock {\em Ann. Inst. H. Poincar\'e Probab. Statist.}, 41(5):871--885,
  2005.

\bibitem{Grimmett-2006}
G.~Grimmett.
\newblock {\em The random-cluster model}, volume 333 of {\em Grundlehren der
  Mathematischen Wissenschaften}.
\newblock Springer-Verlag, Berlin, 2006.

\bibitem{Higuchi-1979}
Y.~Higuchi.
\newblock On some limit theorems related to the phase separation line in the
  two-dimensional {I}sing model.
\newblock {\em Z. Wahrsch. Verw. Gebiete}, 50(3):287--315, 1979.

\bibitem{Ioffe-1998}
D.~Ioffe.
\newblock Ornstein-{Z}ernike behaviour and analyticity of shapes for
  self-avoiding walks on {${\bf Z}^d$}.
\newblock {\em Markov Process. Related Fields}, 4(3):323--350, 1998.

\bibitem{Ioffe-2015}
D.~Ioffe.
\newblock Multidimensional random polymers: a renewal approach.
\newblock In {\em Random walks, random fields, and disordered systems}, volume
  2144 of {\em Lecture Notes in Math.}, pages 147--210. Springer, Cham, 2015.

\bibitem{IOSV}
D.~Ioffe, S.~Ott, Shlosman S., and Y.~Velenik.
\newblock Critical prewetting in the 2d {I}sing model.
\newblock In preparation.

\bibitem{Ioffe+Shlosman+Toninelli-2015}
D.~Ioffe, S.~Shlosman, and F.~L. Toninelli.
\newblock Interaction versus entropic repulsion for low temperature {I}sing
  polymers.
\newblock {\em J. Stat. Phys.}, 158(5):1007--1050, 2015.

\bibitem{miracle1995surface}
S.~Miracle-Sole.
\newblock Surface tension, step free energy, and facets in the equilibrium
  crystal.
\newblock {\em J. Stat. Phys.}, 79(1-2):183--214, 1995.

\bibitem{Ott+Velenik-2018}
S.~Ott and Y.~Velenik.
\newblock Potts models with a defect line.
\newblock {\em Comm. Math. Phys.}, 362(1):55--106, 2018.

\bibitem{Ott+Velenik-2018(2)}
S.~Ott and Y.~Velenik.
\newblock Asymptotics of even-even correlations in the {I}sing model.
\newblock {\em Probab. Theory Related Fields}, 175(1-2):309--340, 2019.

\bibitem{Skor56}
A.~V. Skorokhod.
\newblock Limit theorems for stochastic processes.
\newblock {\em Theory of Probability \& Its Applications}, 1(3):261--290, 1956.

\end{thebibliography}

\end{document}